\documentclass[12pt,reqno]{amsart}
\usepackage{amscd,amssymb,amsmath,color}
\usepackage{stmaryrd}
\usepackage{multirow}
  \usepackage[all]{xy}
  \usepackage{enumerate}
\numberwithin{equation}{section}
\usepackage{amsfonts}
\usepackage{mathtools}
\usepackage{amsthm}
\usepackage[dvipsnames]{xcolor}
\usepackage{color}
\usepackage{hyperref}
\usepackage[pagewise]{lineno} 
\usepackage{braket}
\usepackage{caption}
\usepackage[capposition=top]{floatrow}

\usepackage{float}
\usepackage{todonotes}
\usepackage{blkarray}
\usepackage{pgf,tikz}
\usepackage{mathrsfs}
\usetikzlibrary{arrows}
\definecolor{cqcqcq}{rgb}{0.7529411764705882,0.7529411764705882,0.7529411764705882}
\definecolor{qqzzff}{rgb}{0.,0.6,1.}
\definecolor{wwccqq}{rgb}{0.4,0.8,0.}
\definecolor{ffqqqq}{rgb}{1.,0.,0.}
\usepackage{tikz}
\usetikzlibrary{shapes.geometric, arrows}

\usepackage{todonotes}
\usepackage{pgf,tikz}

\usepackage{footmisc}

\DeclareFontFamily{U}{mathb}{\hyphenchar\font45}
\DeclareFontShape{U}{mathb}{m}{n}{
<-6> mathb5 <6-7> mathb6 <7-8> mathb7
<8-9> mathb8 <9-10> mathb9
<10-12> mathb10 <12-> mathb12
}{}
\DeclareSymbolFont{mathb}{U}{mathb}{m}{n}
\DeclareMathSymbol{\llcurly}{\mathrel}{mathb}{"CE}
\DeclareMathSymbol{\ggcurly}{\mathrel}{mathb}{"CF}

\title[\tiny{Equivariant isomorphism of Quantum Lens Spaces of low dimension}]{Equivariant isomorphism of Quantum Lens Spaces of low dimension}

\author{Søren Eilers}
\address{Department of Mathematical Sciences, University of Copenhagen, Universitetsparken 5, 2100 Copenhagen, Denmark} 
\email{eilers@math.ku.dk} 

\author[Sophie Emma Zegers]{Sophie Emma Zegers}
\address{Delft institute of applied mathematics, Delft University of Technology, P.O. Box 5031, 2600 GA Delft, The Netherlands}
\email{s.e.zegers@tudelft.nl, sophieemmazegers@gmail.com}

\date\today



\begin{document}

\begin{abstract}
The quantum lens spaces form a natural and well-studied class of noncommutative spaces which can be subjected to classification using algebraic invariants by drawing on the fully developed classification theory of unital graph $C^*$-algebras.  We introduce the problem of deciding when two quantum lens spaces are equivariantly isomorphic, and solve it in certain basic cases. As opposed to classification up to isomorphism, we can not appeal to a complete general classification theory in the equivariant case, but by combining existing partial results with an ad hoc analysis we can solve the case with dimension 3 completely, and the case with dimension 5 in the prime case.

Our results can be formulated directly in terms of the parameters defining the  quantum lens spaces, and here occasionally take on a rather complicated form which convinces us that there is a deep underlying explanation for our findings. We complement the fully established partial results with computer experiments that may indicate the way forward.
\end{abstract}

\theoremstyle{plain}
\newtheorem{theorem}{Theorem}[section]
\newtheorem{corollary}[theorem]{Corollary}
\newtheorem{lemma}[theorem]{Lemma}
\newtheorem{proposition}[theorem]{Proposition}
\newtheorem{conjecture}[theorem]{Conjecture}
\newtheorem{comment}[theorem]{Comment}
\newtheorem{problem}[theorem]{Problem}
\newtheorem{remarks}[theorem]{Remarks}
\newtheorem{notation}[theorem]{Notation}

\theoremstyle{definition}
\newtheorem{example}[theorem]{Example}

\newtheorem{definition}[theorem]{Definition}

\newtheorem{remark}[theorem]{Remark}

\newcommand{\Nb}{{\mathbb{N}}}
\newcommand{\Rb}{{\mathbb{R}}}
\newcommand{\Tb}{{\mathbb{T}}}
\newcommand{\Zb}{{\mathbb{Z}}}
\newcommand{\Cb}{{\mathbb{C}}}

\newcommand{\Af}{\mathfrak A}
\newcommand{\Bf}{\mathfrak B}
\newcommand{\Ef}{\mathfrak E}
\newcommand{\Gf}{\mathfrak G}
\newcommand{\Hf}{\mathfrak H}
\newcommand{\Kf}{\mathfrak K}
\newcommand{\Lf}{\mathfrak L}
\newcommand{\Mf}{\mathfrak M}
\newcommand{\Rf}{\mathfrak R}

\newcommand{\mm}{\underline{m}}
\newcommand{\KKK}{\mathbb{K}}
\newcommand{\PP}{\mathcal P}
\newcommand{\x}{\mathfrak x}
\def\C{\mathbb C}
\def\N{\mathbb N}
\def\R{\mathbb R}
\def\T{\mathbb T}
\def\Z{\mathbb Z}

\def\A{{\mathcal A}}
\def\B{{\mathcal B}}
\def\D{{\mathcal D}}
\def\E{{\mathcal E}}
\def\F{{\mathcal F}}
\def\G{{\mathcal G}}
\def\H{{\mathcal H}}
\def\J{{\mathcal J}}
\def\K{{\mathcal K}}
\def\LL{{\mathcal L}}
\def\N{{\mathcal N}}
\def\M{{\mathcal M}}
\def\N{{\mathcal N}}
\def\OO{{\mathcal O}}
\def\P{{\mathcal P}}
\def\Q{{\mathcal Q}}
\def\SS{{\mathcal S}}
\def\U{{\mathcal U}}
\def\W{{\mathcal W}}

\def\ext{\operatorname{Ext}}
\def\span{\operatorname{span}}
\def\clsp{\overline{\operatorname{span}}}
\def\Ad{\operatorname{Ad}}
\def\ad{\operatorname{Ad}}
\def\tr{\operatorname{tr}}
\def\id{\operatorname{id}}
\def\en{\operatorname{End}}
\def\aut{\operatorname{Aut}}
\def\out{\operatorname{Out}}
\def\coker{\operatorname{coker}}
\def\pol{\mathcal{O}}
\renewcommand{\phi}{\varphi}

\newcommand{\SL}{{\text{SL}}}
\newcommand{\DT}{\mathcal{DT}}
\newcommand{\DQ}{\mathcal{DQ}}
\newcommand{\DQp}{\mathcal{DQ}^1}
\newcommand{\sh}{\mathsf{S}}
\newcommand{\perm}{\mathsf{P}}
\newcommand{\oo}{\mathsf{0}}
\newcommand{\iverson}[1]{\left\llbracket#1\right\rrbracket}
\newcommand{\nn}{\underline{n}}
\newcommand{\ZZ}{\mathbb{Z}}
\newcommand{\NN}{\mathbb{N}}
\newcommand{\AAA}{\mathfrak A}
\newcommand{\matrM}{\mathsf M}
\newcommand{\mylat}{\mathcal L}

\def\la{\langle}
\def\ra{\rangle}
\def\rh{\rightharpoonup}
\def\cl{\textcolor{blue}{$\clubsuit$}}

\def\bst{\textcolor{blue}{$\bigstar$}}

\newcommand{\norm}[1]{\left\lVert#1\right\rVert}
\newcommand{\inpro}[2]{\left\langle#1,#2\right\rangle}
\newcommand{\Mod}[1]{\ (\mathrm{mod}\ #1)}
\renewcommand{\thefootnote}{\alph{footnote}}
\renewcommand{\bibname}{References}
\newcommand{\myA}{\mathsf{B}}

\newcommand{\VS}[1][k]{\mbox{$C(S_q^{2{#1}+1})$}}
\newcommand{\ourE}{\mbox{$E_{r;\underline{m}}$}}
\newcommand{\ourEalt}{\mbox{$E_{r;\underline{n}}$}}
\newcommand{\ourF}{\mbox{$F_{r;\underline{m}}$}}
\newcommand{\ourFalt}{\mbox{$F_{r;\underline{n}}$}}

\newcommand{\QLS}[1][2k+1]{\mbox{$C(L_q^{#1}(r;\underline{m}))$}}
\newcommand{\QLSalt}[1][2k+1]{\mbox{$C(L_q^{#1}(r;\underline{n}))$}}
\newcommand{\QLSr}[1][r]{\mbox{$C(L_q^{2k+1}(#1;\underline{m}))$}}
\newcommand{\QLSralt}[1][r]{\mbox{$C(L_q^{2k+1}(#1;\underline{n}))$}}
\newcommand{\Zr}{{\ZZ/r}}
\newcommand{\Zru}{{(\ZZ/r)^\times}}
\newcommand{\myto}{\ggcurly}
\newcommand{\myo}{\succ}
\newcommand{\divisors}{\tau}
\newcommand{\newSS}[2][r]{{\overline{\mathcal W}(#1;#2)}}
\newcommand{\newSStilde}[2][r]{{\mathcal W(#1;#2)}}
\newcommand{\lt}{\operatorname{rt}}

\maketitle

\addtocounter{section}{-1}

\section{introduction}
A well-studied $q$-deformed object in noncommutative geometry is the quantum odd sphere by Vaksman and Soibelman in \cite{llvyv:afqgsodqs} denoted \VS\ with $q\in (0,1)$. It is the universal $C^*$-algebra generated by $k+1$ elements denoted $z_i$ for $i=1,\dots,k+1$ subject to a specific set of relations (see \cite[Section 4]{llvyv:afqgsodqs}). If $q=1$ the $C^*$-algebra $C(S_1^{2k+1})$ is the algebra of continuous functions on the odd sphere. Hence, we can think about the $C^*$-algebra \VS\ as continuous functions on the virtual quantum space $S_q^{2k+1}$. We denote by $d=2k+1$ the dimension of \VS.

In analogue with the classical setting we can for each positive integer $r$ and each sequence of positive integers $\underline{m}=(m_1,m_2,\dots,m_{k+1})$ define, by universality, an action of the finite cyclic group $\Zr$ on \VS\ given by 
$$
\rho_{\underline{m}}^r: z_i\mapsto \theta^{m_i}z_i,
$$
where $\theta=e^{\frac{2\pi i}{r}}$ i.e. a generator of $\Zr$ considered as a subset of $\T$. The \emph{quantum lens space} \QLS\ is defined as the fixed point algebra of \VS\ under this action. Quantum lens spaces have served as interesting examples in noncommutative geometry e.g. they can under certain conditions be described  as the total space of noncommutative line bundles in which their K-theory can be  investigated see e.g. \cite{tbsaf:qt, fasbgl:gsqls, fdgl:qwpls}.

The Vaksman and Soibelman sphere was in \cite{jhhws:qspsga} shown to be a graph $C^*$-algebra by Hong and Szymański. Also quantum lens spaces admit a graph $C^*$-algebraic description; in the case where $\gcd(m_i,r)=1$ for all $i$ this was proven in \cite{jhhws:qlsga}, and later the result was generalised to arbitrary weights in \cite{tbws:cqlwps,tgsez:cdqlsga}. The  graph $C^*$-algebraic description of quantum lens spaces leads  by  \cite{segrerapws:ccuggs}  to the fact that the question of whether any pair of quantum lens spaces are mutually $*$-isomorphic can be reduced to invariants of a $K$-theoretic nature and is in fact decidable by a terminating algorithm. However, translating the relevant condition in $K$-theory to concrete criteria on the parameters defining the quantum lens spaces is difficult in general. 

Such a concrete description was obtained in   \cite{segrerapws:ccuggs} for all quantum lens spaces of dimension less than or equal to $7$ for which $\gcd(m_i,r)=1$, and it was extended in \cite{tgsez:cdqlsga} to the case where we allow $\gcd(m_i,r)\neq 1$ for precisely one $i$.  In higher dimensions, again with   $\gcd(m_i,r)=1$ for all $i$, the authors of \cite{pljfrkpmrr:ccqls} find for all $r$ the smallest dimension at which two non-isomorphic quantum lens spaces exist, and provide partial evidence for a formula for the  general number of different isomorphism classes.

In the present paper we consider quantum lens spaces of  low dimension ($d=3$ or $d=5$) which from the graph $C^*$-algebraic description are easily seen all to be isomorphic. To obtain more information about the structure of the quantum lens spaces, we therefore investigate under which conditions on the weights, it is possible to have an isomorphism that preserves the natural circle action on the quantum lens space inherited from the Vaksman and Soibelman quantum sphere. For this, our starting point is the description of the quantum sphere as a graph $C^*$-algebra under which the circle action becomes the gauge action on the graph $C^*$-algebra. It is important to remark that we cannot pass to the graph $C^*$-algebraic description of the quantum lens spaces, since the circle action is not preserved under this isomorphism. The main results, Theorems \ref{mainresult3} and \ref{mainresult5}, give a complete description of when there exists an isomorphism preserving the circle action between two quantum lens spaces of dimension $3$, and resolves the same question at dimension $5$ when the main parameter is a prime. 

Moreover, in all cases when such an isomorphism is found to exist we show how to construct it explicitly. We 
relate our work to a conjecture by Hazrat, and explain how we can use computer algebra systems to make further predictions and conjectures. Indeed, also the results that we by now have a complete understanding of were first found by computations in Maple.

The paper is structured as follows. We first recall the description of quantum lens spaces as graph $C^*$-algebras. In Section 2 we describe the fixed point algebra of the quantum lens spaces under the gauge action as being Morita equivalent to another graph $C^*$-algebra. We also describe the lattice of gauge ideals inside the fixed point algebra. Section 3 is devoted to the lowest dimensional quantum lens spaces, dimension 3, in which we come up with a number theoretic invariant for equivariant isomorphism. In the lowest dimension it turns out that equivariant isomorphism is completely determined by the ideal structure of the fixed point algebra. This is not the case for higher dimensions. Hence, in Section 4 we describe the \emph{dimension triple} which we will use in order to come up with an invariant for dimension 5 in the case where the order of the acting group is a prime number, see Section 5 and 6.
 In Section 7, we discuss our expectations for other cases, and present some conjectures based on computer experiments.

\subsection*{Acknowledgements}
The first named author gratefully acknowledges support from Independent Research Fund Denmark (DFF), Research Project 2 ``Operator Algebras, Groups, and Quantum Spaces'' as well as from  EU Staff Exchange project 101086394 ``Operator Algebras That
One Can See''.
The second author was supported by the Carlsberg Foundation through an Internationalisation Fellowship.

We thank Morten S. Risager and Efren Ruiz for helpful conversations.

\subsection{Preliminaries and notation}
We first recall some concepts of graph $C^*$-algebras which are needed in this paper. A directed graph $E=(E^0,E^1,r,s)$ consists of a countable set $E^0$ of \textit{vertices}, a countable set $E^1$ of \textit{edges} and two maps $r,s: E^1\to E^0$ called the \textit{range map} and the \textit{source map} respectively. For an edge $e\in E^1$ from $v$ to $w$ we have $s(e)=v$ and $r(e)=w$. 
A \textit{path} $\alpha$ in a graph is a finite sequence $\alpha=e_1e_2\cdots e_n$ of edges satisfying $r(e_i)=s(e_{i+1})$ for $i=1,...,n-1$. 

The graph $C^*$-algebra of a directed graph is defined as follows (following \cite{njfmlir:cig}, see e.g.\ \cite{segrerapws:ccuggs} and the references given there).
\begin{definition}\label{graphalgebra}
Let $E=(E^0,E^1,r,s)$ be a directed graph. The graph $C^*$-algebra $C^*(E)$
is the universal $C^*$-algebra generated by families of orthogonal projections $\{P_v : \ v\in E^0\}$ and partial isometries $\{S_e : \ e\in E^1\}$ with mutually orthogonal ranges (i.e. $S_e^*S_f=0, e\neq f$) 
subject to the relations 
\begin{itemize}
\item[] (CK1) $S_e^*S_e=P_{r(e)}$
\item[] (CK2) $S_eS_e^*\leq P_{s(e)}$
\item[] (CK3) $P_v=\underset{s(e)=v}{\sum}S_e S_e^*$, if $\{e\in E^1: \ s(e)=v\}$ is finite and nonempty. 
\end{itemize}
\end{definition}
For a finite path $\alpha=e_1e_2\cdots e_n$ we let $S_\alpha=S_{e_1}S_{e_2}\cdots S_{e_n}$ which is a partial isometry. A subset $H\subseteq E^0$ is \textit{hereditary} if it satisfies the condition: If $v\in H$ and $u\in E^0$ is such that $v\geq u$ (i.e. there exists a path from $v$ to $u$) then $u\in H$. It is called \textit{saturated} if it satisfies the condition: If $w\in E^0$ with $0<|s^{-1}(w)|<\infty$ and for each $e\in E^1$, for which $s(e)=w$ we have $r(e)\in H$, then $w\in H$. 

We can by universality define a circle action, called the \textit{gauge action}, $\gamma: \mathbb{T}\to \text{Aut}(C^*(E))$ for which 
$\gamma_z(p_v)=p_v \ \text{and} \ \gamma_z(s_e)=zs_e$
for all $v\in E^0, e\in E^1$ and $z\in \mathbb{T}$.
\\

Moreover, we remark the following:
\begin{itemize}
\item We will  use \textit{mod} in two different ways throughout the text. When writing ``$a\equiv b \pmod{r}$'' we refer to the statement  $r\mid (a-b)$ and when writing ``$a \mod{r}$'' we are referring to the modulo operation which calculates the remainder of $a$ divided by $r$. 
\item By $\iverson{-}$ we refer to the Iverson bracket given on a statement $A$ by
$$
\iverson{A}=\begin{cases}
1 & \text{if $A$ is true}, \\
0 & \text{otherwise}.
\end{cases}
$$
\end{itemize}
\section{Quantum lens spaces}
The odd quantum sphere \VS\ is in \cite{jhhws:qspsga} shown to be isomorphic to the graph $C^*$-algebra $C^*(L_{2k+1})$ by an explicit isomorphism.  The graph $L_{2k+1}$ has $k+1$ vertices $v_i, i=0,\dots,n$ and edges $e_{ij}, 1\leq i\leq j\leq k+1$ such that  $s(e_{ij})=v_i, r(e_{ij})=v_j$. 


\begin{figure}[H]
\centering
\begin{tikzpicture}[scale=1.8]
\draw (0,0) -- (2,0);
\filldraw [black] (0,0) circle (1pt);
\filldraw [black] (2,0) circle (1pt);
\draw [->] (0,0) -- (1,0);

\draw (0,0.3) circle [radius=0.3cm];
\draw (2,0.3) circle [radius=0.3cm];
\draw [->] (-0.01,0.6) -- (0.01,0.6);
\draw [->] (1.99,0.6) -- (2.01,0.6);

\node at (0, 0.2)  {$v_1$};
\node at (2, 0.2)  {$v_2$};
\node at (1, 0.2)  {$e_{12}$};
\node at (0, 0.75)  {$e_{11}$};
\node at (2, 0.75)  {$e_{22}$};
\end{tikzpicture}
\caption{The graph $L_{3}$}
\captionsetup{aboveskip=0pt,font=it}
\label{quantum3sphere}
\end{figure}
\begin{figure}[H]
\centering
\begin{tikzpicture}[scale=2]
\draw (0,0) -- (4,0);
\filldraw [black] (0,0) circle (1pt);
\filldraw [black] (2,0) circle (1pt);
\filldraw [black] (4,0) circle (1pt);
\draw [->] (0,0) -- (1,0);
\draw [->] (2,0) -- (3,0);

\draw[] (0,0) to [out=-20,in=-160] (4,0);
\draw [->] (1.99,-0.4) -- (2,-0.4);

\draw (0,0.3) circle [radius=0.3cm];
\draw (2,0.3) circle [radius=0.3cm];
\draw (4,0.3) circle [radius=0.3cm];
\draw [->] (-0.01,0.6) -- (0.01,0.6);
\draw [->] (1.99,0.6) -- (2.01,0.6);
\draw [->] (3.99,0.6) -- (4.01,0.6);

\node at (0, 0.2)  {$v_1$};
\node at (2, 0.2)  {$v_2$};
\node at (4, 0.2)  {$v_3$};
\node at (1, 0.2)  {$e_{12}$};
\node at (3, 0.2)  {$e_{23}$};
\node at (2, -0.25)  {$e_{13}$};
\node at (0, 0.75)  {$e_{11}$};
\node at (2, 0.75)  {$e_{22}$};
\node at (4, 0.75)  {$e_{33}$};
\end{tikzpicture}
\caption{The graph $L_{5}$}
\captionsetup{aboveskip=0pt,font=it}
\label{quantum5sphere}
\end{figure}
The canonical action $\rho_{\underline{m}}^r$ on \VS\ translates under the isomorphism with $C^*(L_{2k+1})$, to the  action
$$
S_{e_{ij}}\mapsto \theta^{m_i}S_{e_{ij}}, \ \ P_{v_i}\mapsto P_{v_i}, 
$$
which we also denote by $\rho_{\underline{m}}^r$.

The Vaksman and Soibelman sphere admits by universality a natural circle action given on the generators $z_i, i=1,2,\dots,k+1$ by $z_i\mapsto \omega z_i, \omega\in\T, i=1,\dots,k+1$ which under the isomorphism with $C^*(L_{2k+1})$ becomes the gauge action on the graph $C^*$-algebra i.e. 
$$
S_{e_{ij}}\mapsto \omega S_{e_{ij}}, P_{v_i}\mapsto P_{v_i}.
$$
Moreover, we have a natural circle action on $C^*(L_3)^{\rho^r_{\underline{m}}}$ which is inherited from the one on the Vaksman and Soibelman sphere, and denote it by $\gamma$. 

By a result of  Crisp (see \cite[Theorem 4.6]{tc:cga})  the fixed point algebra $C^*(L_3)^{\rho^r_{\underline{m}}}$ can be described as a corner of the graph $C^*$-algebra of the \emph{skew product graph}. We recall here the construction. 
\begin{definition}
Let $c_{\mm}:e_{ij}\to m_i \pmod{r}$ be the labelling induced by the action $\rho^{r}_{\underline{m}}$. The skew product graph $L_{2k+1}\times_{c_{\mm}} \Zb$ has $(k+1)r$ vertices denoted $(v_i,\ell),i=1,2,\dots,k+1, \ell=0,1,\dots,r-1$ and edges $(e_{ij},\ell), 1\leq i,j\leq k+1, \ell=0,1,\dots,r-1$ with source and range as follows: 
$$
s((e_{ij}, \ell))=(v_i,\ell-m_i \Mod{r}), \ \ r((e_{ij},\ell))=(v_j,\ell). 
$$
We write  $\ourE$ for this graph.
\end{definition}

We visualise the graph $\ourE$ as having $k+1$ levels, in the first level we line up the vertices $(v_1,\ell), \ell=0,1,\dots,r-1$ in the second one $(v_2,\ell), \ell=0,1,\dots,r-1$ and so on. In each level we obtain one cycle if $\gcd(m_i,r)=1$ for $i=1,2,\dots,k+1$. There are only edges from the lower levels to the higher ones and not in the other direction. The skew product graph is visualised in Figure \ref{skewproduct} for $r=5$ and $\underline{m}=(1,3)$.  


\begin{figure}[H]
\begin{tikzpicture}[scale=1.2]
\begin{scriptsize}
\filldraw [black] (0,0) circle (1pt);
\filldraw [black] (2,0) circle (1pt);
\filldraw [black] (4,0) circle (1pt);
\filldraw [black] (6,0) circle (1pt);
\filldraw [black] (8,0) circle (1pt);

\node at (-0.5, 0)  {$(v_1,0)$};
\node at (1.3, 0)  {$(v_1,1)$};
\node at (3.3, 0)  {$(v_1,2)$};
\node at (5.3, 0)  {$(v_1,3)$};
\node at (7.3, 0)  {$(v_1,4)$};

\draw[->] (0.1,0.1) to [out=20,in=160] (1.9,0.1);
\draw[->] (2.1,0.1) to [out=20,in=160] (3.9,0.1);
\draw[->] (4.1,0.1) to [out=20,in=160] (5.9,0.1);
\draw[->] (6.1,0.1) to [out=20,in=160] (7.9,0.1);
\draw[->] (8,0.2) to [out=165,in=15] (0,0.2);

\filldraw [black] (0,-2) circle (1pt);
\filldraw [black] (2,-2) circle (1pt);
\filldraw [black] (4,-2) circle (1pt);
\filldraw [black] (6,-2) circle (1pt);
\filldraw [black] (8,-2) circle (1pt);

\node at (-0.5, -2)  {$(v_2,0)$};
\node at (1.3, -2)  {$(v_2,1)$};
\node at (3.3, -2)  {$(v_2,2)$};
\node at (5.3, -2)  {$(v_2,3)$};
\node at (7.3, -2)  {$(v_2,4)$};

\draw[->] (0.1,-1.9) to [out=10,in=170] (5.9,-1.9);
\draw[->] (2.1,-1.9) to [out=10,in=170] (7.9,-1.9);
\draw[->] (4,-2.1) to [out=-160,in=-20] (0.1,-2.1);
\draw[->] (5.9,-2.1) to [out=-160,in=-20] (2.1,-2.1);
\draw[->] (7.9,-2.1) to [out=-160,in=-20] (4.1,-2.1);

\draw[->] (0.1,-0.1) to (1.99,-1.9);
\draw[] [->] (2.1,-0.1) to (3.99,-1.9);
\draw [->] (4,-0.1) to (5.99,-1.9);
\draw[] [->] (6,-0.1) to (7.99,-1.9);
\draw[] [->] (8,-0.1) to (0,-1.9);
\end{scriptsize}
\end{tikzpicture}
\captionsetup{aboveskip=0pt,font=it}
\caption{The graph $E_{5;(1,3)}$}
\label{skewproduct}
\end{figure}

The description of the quantum lens space \QLS\ as the corner of the graph $C^*$-algebra of the \emph{skew product graph} is as follows:  
\begin{equation}\label{thecorner}
\QLS\cong C^*(L_{2k+1})^{\rho_{\underline{m}}^r}\cong \sum_{i=1}^{k+1}P_{(v_i,0)}C^*(\ourE)\sum_{i=1}^{k+1}P_{(v_i,0)}
\end{equation}
Note that the above isomorphism preserves the circle action as follows from inspection of  \cite{jhhws:qlsga}. 

\begin{notation}\label{not:equiviso}
We write $(\QLS,\gamma)\simeq (\QLSalt,\gamma)$ (sometimes abbreviated $\underline{m}\simeq_{\gamma} \underline{n}$ when the context is clear)
if there exists an isomorphism between $C(L^{2k+1}_q(r;\underline{m}))$ and $C(L_q^{2k+1}(r;\underline{n})$ that preserves the circle action described above.
\end{notation}

We formulate our first result:

\begin{lemma}\label{Lemma:simpleiso}
For all choices of $r$, $k$ and $\mm$ with $m_i\in\Zru$, when $\alpha \in \Zru$ is given, we get
\[
(\QLS,\gamma)\simeq(C(L_q^{2k+1}(r;\alpha\mm)),\gamma)
\]
\end{lemma}
\begin{proof}
Recall that \QLS\ is defined as the fixed point algebra of \VS\ under the action 
$$
\rho_{\underline{m}}^r: z_i\mapsto \theta^{m_i}z_i
$$
of $\Zr$. Let $\alpha\in  \Zru$, then $\theta^{\alpha}$ is also a generator of $\Zr$ and we can define a  $\Zr$-action on \VS\ as follows: 
$$
\beta_{\underline{m}}^r: z_i\mapsto (\theta^{\alpha})^{m_i} z_i.
$$  
It follows that 
$$
\rho_{\underline{m}}^r(a)=a \Longleftrightarrow \beta_{\underline{m}}^r(a)=a
$$
Hence $C(S_q^{2n+1})^{\rho_{\underline{m}}^r}=C(S_q^{2n+1})^{\beta_{\underline{m}}^r}$ and therefore the circle action $\gamma$ is clearly preserved between the two fixed point algebras. Since the action $\beta_{\underline{m}}^r$ is the same as $\rho_{\alpha\underline{m}}^r$ we obtain that $(\QLS,\gamma)\simeq(C(L_q^{2k+1}(r,\alpha\mm)),\gamma)$. 
\end{proof}

In \cite{jhhws:qlsga} it was shown that quantum lens spaces are graph $C^*$-algebras by showing that the corner in \eqref{thecorner} is itself a graph $C^*$-algebra of a certain graph $L_{2k+1}^{r;\underline{m}}$. This result also proves that all the quantum lens spaces with the same $r$ are indeed mutually isomorphic in dimension $3$ and $5$, see Figure \ref{graphlensspace} and \ref{graphlensspace5}. Note the notation ``$(n)$'' to indicate in the depicted graph that the arrow represents $n$ parallel edges.

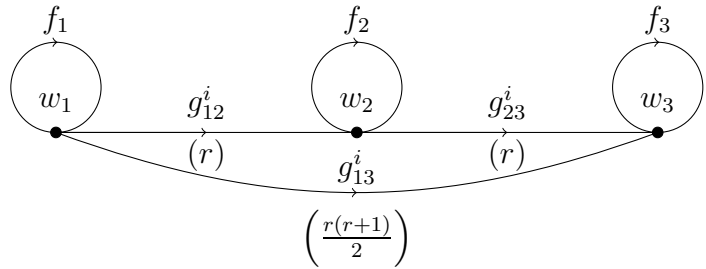
\begin{figure}[H]
    \centering
\begin{tikzpicture}[scale=2]
\draw (0,0) -- (2,0);
\filldraw [black] (0,0) circle (1pt);
\filldraw [black] (2,0) circle (1pt);
\draw [->] (0,0) -- (1,0);

\draw (0,0.3) circle [radius=0.3cm];
\draw (2,0.3) circle [radius=0.3cm];
\draw [->] (-0.01,0.6) -- (0.01,0.6);
\draw [->] (1.99,0.6) -- (2.01,0.6);

\draw[] (0,0) to [out=-10,in=-170] (2,0);
\draw[] (0,0) to [out=-20,in=-160] (2,0);
\draw[] (0,0) to [out=-50,in=-130] (2,0);
\draw [->] (0.99,-0.1) -- (1,-0.1);
\draw [->] (0.99,-0.2) -- (1,-0.2);
\draw [->] (0.99,-0.45) -- (1,-0.45);

\node at (0, 0.2)  {$w_1$};
\node at (2, 0.2)  {$w_2$};
\node at (1, 0.2)  {$g_{1}$};
\node at (1, -0.6)  {$g_{r}$};
\node at (0.99, -0.27)  {$\vdots$};
\node at (0, 0.75)  {$f_{1}$};
\node at (2, 0.75)  {$f_{2}$};
\end{tikzpicture}
\caption{The graph $L_{3}^{r;\underline{m}}$}
\captionsetup{aboveskip=0pt,font=it}
\label{graphlensspace}
\end{figure}

\begin{figure}[H]
\centering
\begin{tikzpicture}[scale=2]
\draw (0,0) -- (4,0);
\filldraw [black] (0,0) circle (1pt);
\filldraw [black] (2,0) circle (1pt);
\filldraw [black] (4,0) circle (1pt);
\draw [->] (0,0) -- (1,0);
\draw [->] (2,0) -- (3,0);

\draw[] (0,0) to [out=-20,in=-160] (4,0);
\draw [->] (1.99,-0.4) -- (2,-0.4);

\draw (0,0.3) circle [radius=0.3cm];
\draw (2,0.3) circle [radius=0.3cm];
\draw (4,0.3) circle [radius=0.3cm];
\draw [->] (-0.01,0.6) -- (0.01,0.6);
\draw [->] (1.99,0.6) -- (2.01,0.6);
\draw [->] (3.99,0.6) -- (4.01,0.6);

\node at (0, 0.2)  {$w_1$};
\node at (2, 0.2)  {$w_2$};
\node at (4, 0.2)  {$w_3$};
\node at (1, -0.15)  {$(r)$};
\node at (1, 0.2)  {$g_{12}^i$};
\node at (3, -0.15)  {$(r)$};
\node at (3, 0.2)  {$g_{23}^i$};
\node at (2, -0.7)  {$\left( \frac{r(r+1)}{2}\right)$};
\node at (2, -0.23) {$g_{13}^i$};
\node at (0, 0.75)  {$f_{1}$};
\node at (2, 0.75)  {$f_{2}$};
\node at (4, 0.75)  {$f_{3}$};
\end{tikzpicture}
\caption{The graph $L_{5}^{r;\underline{m}}$}
\captionsetup{aboveskip=0pt,font=it}
\label{graphlensspace5}
\end{figure}

The graph $L_{2k+1}^{r;\underline{m}}$ is constructed from the skew product graph by considering a specific type of paths called \textit{admissible paths}. The number of paths between $w_1$ and $w_2$ is then the number of admissible paths from $(v_1,0)$ to $(v_2,0)$ i.e. paths which does not pass through these two vertices. 
The isomorphism between the corner in \eqref{thecorner} and the graph $C^*$-algebra of $L_{2k+1}^{r;\underline{m}}$ is given explicitly in \cite{jhhws:qlsga} and maps an admissible path from $(v_i,0)$ to $(v_j,0)$ in $E_{r;\mm}$ to one of the edges in $L_{2k+1}^{r;\underline{m}}$ with source $w_i$ and range $w_j$. Hence it is clear that the gauge action is not preserved under this isomorphism. 

Since we are aiming to classify up to equivariant isomorphism, we cannot work with this picture and do not describe the graph $L_{2k+1}^{r;\underline{m}}$ in detail in general. Instead we use the description as the corner of the skew product graph. We emphasize that even through we cannot use the graph $C^*$-algebraic picture in \cite{jhhws:qlsga} their approach is still important for our work since it relies heavily on the notion of admissible paths.

We note that it is in fact possible to prove, cf. \cite{gaermt:megr}, that \QLS\ cannot be equivariantly isomorphic to any graph $C^*$-algebra. This justifies why we will work with the description of quantum lens spaces as the fixed point algebra of $	C^*(L_{2k+1})$ or as the corner \eqref{thecorner}. 
\\

We now introduce sets of specific paths inside $L_{2k+1}$ which we in dimension 3 and 5 will see corresponds to a set of generators of the fixed point algebra $C^*(L_{2k+1})^{\rho_{\underline{m}}^r}$ describing a quantum lens space (see Proposition \ref{PropGenerators3} and \ref{PropGenerators5}). We also introduce two multisets which become crucial in the classification up to equivariant isomorphism. 

For a positive integer $t$ let $e_{ii}^t$ denote the path in the graph $L_{2k+1}$ where $e_{ii}$ has been repeated $t$ times.
\begin{definition}\label{def:multisets} 
Let $\underline{m}=(m_1,m_2,\dots,m_{k+1})$ with $\gcd(m_i,r)=1, i=1,...,k+1$. We define $\mathcal{A}(r;(m_{i_1},m_{i_2},...,m_{i_{\ell}}))$ for $i_1<i_2<...<i_{\ell}$ to be the paths in $L_{2k+1}$ of the form 
$$
\alpha_{t_{i_1},...,t_{i_{\ell}}}:=e_{i_1i_1}^{t_{i_1}}e_{i_1i_2}e_{i_2i_2}^{t_{i_2}}e_{i_2i_3}e_{i_3i_3}^{t_{i_3}}e_{i_3i_4}\cdots e_{i_\ell i_\ell}^{t_{i_{\ell}}}
$$ 
with $t_{i_j}\in \{0,1,...,r-1\}$ for which 
\begin{equation}\label{InFixedPA}
m_{i_1}(t_{i_1}+1)+m_{i_2}(t_{i_2}+1)+\cdots +m_{i_{\ell}}t_{i_{\ell}}\equiv 0 \pmod{r}
\end{equation}
and such that $\alpha_{t_{i_1},...,t_{i_{\ell}}}$ cannot be written as a combination of shorter paths from sets $\mathcal{A}(r,(m_{j_1},...,m_{j_s}))$ with $\{j_1,...,j_s\}\subset \{i_1,...,i_{\ell}\}$. 
\\
Moreover, we define the multisets
$$
\begin{aligned}
&\newSStilde{(m_{i_1},m_{i_2},...,m_{i_\ell})}:=\{|\alpha|:  \alpha\in \mathcal{A}(r;(m_{i_1},m_{i_2},...,m_{i_\ell}) \}
\\
&\newSS{(m_{i_1},m_{i_2},...,m_{i_\ell})}:=\{|\alpha| \mod{r}:  \alpha\in \mathcal{A}(r;(m_{i_1},m_{i_2},...,m_{i_\ell}) \}\subseteq \ZZ/r
\end{aligned}
$$
\end{definition}

Note that the effect of condition \eqref{InFixedPA} is to ensure that the partial isometry corresponding to $\alpha_{t_{i_1},...,t_{i_{\ell}}}$ lies inside the fixed point algebra $C^*(L_{2k+1})^{\rho_{\underline{m}}^r}$. 
\begin{notation}\label{WreducedTheSame}
For $\mm=(m_1,m_2,\dots,m_{k+1})$ and $\nn=(n_1,n_2,\dots,n_{k+1})$  with $\gcd(m_i,r)=\gcd(n_i,r)=1$ for $i=1,...,k+1$ we write $\newSS{\underline{m}}=\newSS{\underline{n}}$ if for each pair $s,t\in \{1,...,k+1\}$ with $s<t$ we have 
$$
\bigcup_{l=0}^{t-2}\left(\bigcup_{s<i_1<...<i_{\ell}<t}  \newSS{(m_{s},m_{i_1},...,m_{i_{\ell}},m_{t})}\right) = \bigcup_{l=0}^{t-2}\left(\bigcup_{s<i_1<...<i_{\ell}<t} \newSS{(n_{s},n_{i_2},...,n_{i_{\ell}}, n_{t})}\right)
$$ 
\end{notation}
Intuitively,  $\newSS{\underline{m}}=\newSS{\underline{n}}$  means that the multiset consisting of the length modulo $r$ of all the paths starting in $v_s$ and ending in $v_t$ is the same for the two sets of weights.

\begin{corollary}\label{Cor:SameMultiSet}
For all choices of $r$, $k$ and $\mm$ with $m_i\in\Zru$, when $\alpha \in \Zru$ is given, we get
$$
\newSS{\underline{m}}=\newSS{\alpha\underline{m}}.
$$
\end{corollary}
\begin{proof}
This follows directly by Lemma \ref{Lemma:simpleiso} since it is shown that the two fixed point algebras agrees, hence the paths inside $\mathcal{A}(r;-)$ all agrees. 
\end{proof}

\section{Structure of the fixed point algebra}
We will in this section give a description of the fixed point algebra of the quantum lens spaces under the canonical circle action. This is a key component of our analysis, and as we will note at the end of the section, establishes a complete classification at arbitrary dimension for $r=3,4,6,12$.

\begin{definition}
Let $E$ be a directed graph with finitely many vertices. We define $E\times_1\Zb$ to be the graph with vertices $E^0\times\Zb$, edges $E^1\times\Zb$ and with range and source maps given as follows: 
$$
s(e,n)=(s(e),n-1) \ \text{and} \ r(e,n)=(r(e),n)
$$
\end{definition}
We are interested in a particular subgraph of $\ourE\times_1\Zb$, denoted $\ourF$, namely the graph for which the vertices equals the hereditary subset of $\ourE^0\times\Zb$ generated by the vertices $((v_i,0),0)$ for $i=1,2,\dots,{k+1}$ and the set of edges is
$$
(\ourF)^1:=\{e\in \ourE\times_1\Zb: \ s(e)\in (\ourF)^0\}. 
$$

\begin{example}\label{exgraph1}
We consider the case where $r=5$ and $\underline{m}=(1,3)$. 
\begin{figure}[H]
\begin{tikzpicture}[scale=1]
\begin{scriptsize}

\node at (-1, 0)  {$((v_1,0),0)$};
\node at (2, 0.5)  {$((v_1,0),1)$};
\node at (4, 0.5)  {$((v_1,0),2)$};
\node at (6, 0.5)  {$((v_1,0),3)$};
\node at (8, 0.5)  {$((v_1,0),4)$};
\node at (10, 0.5)  {$((v_1,0),5)$};

\node at (-1, -1)  {$((v_1,1),0)$};
\node at (-1, -2)  {$((v_1,2),0)$};
\node at (-1, -3)  {$((v_1,3),0)$};
\node at (-1, -4)  {$((v_1,4),0)$};

\filldraw [black] (0,0) circle (1pt);
\filldraw [black] (2,0) circle (1pt);
\filldraw [black] (4,0) circle (1pt);
\filldraw [black] (6,0) circle (1pt);
\filldraw [black] (8,0) circle (1pt);
\filldraw [black] (10,0) circle (1pt);

\filldraw [black] (0,-1) circle (1pt);
\filldraw [black] (2,-1) circle (1pt);
\filldraw [black] (4,-1) circle (1pt);
\filldraw [black] (6,-1) circle (1pt);
\filldraw [black] (8,-1) circle (1pt);
\filldraw [black] (10,-1) circle (1pt);

\filldraw [black] (0,-2) circle (1pt);
\filldraw [black] (2,-2) circle (1pt);
\filldraw [black] (4,-2) circle (1pt);
\filldraw [black] (6,-2) circle (1pt);
\filldraw [black] (8,-2) circle (1pt);
\filldraw [black] (10,-2) circle (1pt);

\filldraw [black] (0,-3) circle (1pt);
\filldraw [black] (2,-3) circle (1pt);
\filldraw [black] (4,-3) circle (1pt);
\filldraw [black] (6,-3) circle (1pt);
\filldraw [black] (8,-3) circle (1pt);
\filldraw [black] (10,-3) circle (1pt);

\filldraw [black] (0,-4) circle (1pt);
\filldraw [black] (2,-4) circle (1pt);
\filldraw [black] (4,-4) circle (1pt);
\filldraw [black] (6,-4) circle (1pt);
\filldraw [black] (8,-4) circle (1pt);
\filldraw [black] (10,-4) circle (1pt);

\draw[cyan] [->] (0,0) -- (2,-1);
\draw[cyan] [->] (2,-1) -- (4,-2);
\draw[cyan] [->] (4,-2) -- (6,-3);
\draw[cyan] [->] (6,-3) -- (8,-4);
\draw[cyan] [->] (8,-4) -- (10,0);

\draw[cyan] [->] (0,-1) -- (2,-2);
\draw[cyan] [->] (2,-2) -- (4,-3);
\draw[cyan] [->] (4,-3) -- (6,-4);
\draw[cyan] [->] (6,-4) -- (8,0);
\draw[cyan] [->] (8,0) -- (10,-1);

\draw[cyan] [->] (0,-2) -- (2,-3);
\draw[cyan] [->] (2,-3) -- (4,-4);
\draw[cyan] [->] (4,-4) -- (6,0);
\draw[cyan] [->] (6,0) -- (8,-1);
\draw[cyan] [->] (8,-1) -- (10,-2);

\draw[cyan] [->] (0,-3) -- (2,-4);
\draw[cyan] [->] (2,-4) -- (4,0);
\draw[cyan] [->] (4,0) -- (6,-1);
\draw[cyan] [->] (6,-1) -- (8,-2);
\draw[cyan] [->] (8,-2) -- (10,-3);

\draw[cyan] [->] (0,-4) -- (2,0);
\draw[cyan] [->] (2,0) -- (4,-1);
\draw[cyan] [->] (4,-1) -- (6,-2);
\draw[cyan] [->] (6,-2) -- (8,-3);
\draw[cyan] [->] (8,-3) -- (10,-4);

\draw[] [dotted] (-1.5,-4.5) -- (11,-4.5);

\draw[orange] [->] (0,0) -- (2,-6);
\draw[orange] [->] (2,-1) -- (4,-7);
\draw[orange] [->] (4,-2) -- (6,-8);
\draw[orange] [->] (6,-3) -- (8,-9);
\draw[orange] [->] (8,-4) -- (10,-5);

\draw[orange] [->] (0,-1) -- (2,-7);
\draw[orange] [->] (2,-2) -- (4,-8);
\draw[orange] [->] (4,-3) -- (6,-9);
\draw[orange] [->] (6,-4) -- (8,-5);
\draw[orange] [->] (8,0) -- (10,-6);

\draw[orange] [->] (0,-2) -- (2,-8);
\draw[orange] [->] (2,-3) -- (4,-9);
\draw[orange] [->] (4,-4) -- (6,-5);
\draw[orange] [->] (6,0) -- (8,-6);
\draw[orange] [->] (8,-1) -- (10,-7);

\draw[orange] [->] (0,-3) -- (2,-9);
\draw[orange] [->] (2,-4) -- (4,-5);
\draw[orange] [->] (4,0) -- (6,-6);
\draw[orange] [->] (6,-1) -- (8,-7);
\draw[orange] [->] (8,-2) -- (10,-8);

\draw[orange] [->] (0,-4) -- (2,-5);
\draw[orange] [->] (2,0) -- (4,-6);
\draw[orange] [->] (4,-1) -- (6,-7);
\draw[orange] [->] (6,-2) -- (8,-8);
\draw[orange] [->] (8,-3) -- (10,-9);

\node at (-1, -5)  {$((v_2,0),0)$};
\node at (-1, -6)  {$((v_2,1),0)$};
\node at (-1, -7)  {$((v_2,2),0)$};
\node at (-1, -8)  {$((v_2,3),0)$};
\node at (-1, -9)  {$((v_2,4),0)$};

\filldraw [black] (0,-5) circle (1pt);
\filldraw [black] (2,-5) circle (1pt);
\filldraw [black] (4,-5) circle (1pt);
\filldraw [black] (6,-5) circle (1pt);
\filldraw [black] (8,-5) circle (1pt);
\filldraw [black] (10,-5) circle (1pt);

\filldraw [black] (0,-6) circle (1pt);
\filldraw [black] (2,-6) circle (1pt);
\filldraw [black] (4,-6) circle (1pt);
\filldraw [black] (6,-6) circle (1pt);
\filldraw [black] (8,-6) circle (1pt);
\filldraw [black] (10,-6) circle (1pt);

\filldraw [black] (0,-7) circle (1pt);
\filldraw [black] (2,-7) circle (1pt);
\filldraw [black] (4,-7) circle (1pt);
\filldraw [black] (6,-7) circle (1pt);
\filldraw [black] (8,-7) circle (1pt);
\filldraw [black] (10,-7) circle (1pt);

\filldraw [black] (0,-8) circle (1pt);
\filldraw [black] (2,-8) circle (1pt);
\filldraw [black] (4,-8) circle (1pt);
\filldraw [black] (6,-8) circle (1pt);
\filldraw [black] (8,-8) circle (1pt);
\filldraw [black] (10,-8) circle (1pt);

\filldraw [black] (0,-9) circle (1pt);
\filldraw [black] (2,-9) circle (1pt);
\filldraw [black] (4,-9) circle (1pt);
\filldraw [black] (6,-9) circle (1pt);
\filldraw [black] (8,-9) circle (1pt);
\filldraw [black] (10,-9) circle (1pt);

\draw[MidnightBlue] [->] (0,-5) -- (2,-8);
\draw[MidnightBlue] [->] (2,-8) -- (4,-6);
\draw[MidnightBlue] [->] (4,-6) -- (6,-9);
\draw[MidnightBlue] [->] (6,-9) -- (8,-7);
\draw[MidnightBlue] [->] (8,-7) -- (10,-5);

\draw[MidnightBlue] [->] (0,-6) -- (2,-9);
\draw[MidnightBlue] [->] (2,-9) -- (4,-7);
\draw[MidnightBlue] [->] (4,-7) -- (6,-5);
\draw[MidnightBlue] [->] (6,-5) -- (8,-8);
\draw[MidnightBlue] [->] (8,-8) -- (10,-6);

\draw[MidnightBlue] [->] (0,-7) -- (2,-5);
\draw[MidnightBlue] [->] (2,-5) -- (4,-8);
\draw[MidnightBlue] [->] (4,-8) -- (6,-6);
\draw[MidnightBlue] [->] (6,-6) -- (8,-9);
\draw[MidnightBlue] [->] (8,-9) -- (10,-7);

\draw[MidnightBlue] [->] (0,-8) -- (2,-6);
\draw[MidnightBlue] [->] (2,-6) -- (4,-9);
\draw[MidnightBlue] [->] (4,-9) -- (6,-7);
\draw[MidnightBlue] [->] (6,-7) -- (8,-5);
\draw[MidnightBlue] [->] (8,-5) -- (10,-8);

\draw[MidnightBlue] [->] (0,-9) -- (2,-7);
\draw[MidnightBlue] [->] (2,-7) -- (4,-5);
\draw[MidnightBlue] [->] (4,-5) -- (6,-8);
\draw[MidnightBlue] [->] (6,-8) -- (8,-6);
\draw[MidnightBlue] [->] (8,-6) -- (10,-9);

\node at (10.5, -0)  {$\cdots$};
\node at (10.5, -1)  {$\cdots$};
\node at (10.5, -2)  {$\cdots$};
\node at (10.5, -3)  {$\cdots$};
\node at (10.5, -4)  {$\cdots$};
\node at (10.5, -5)  {$\cdots$};
\node at (10.5, -6)  {$\cdots$};
\node at (10.5, -7)  {$\cdots$};
\node at (10.5, -8)  {$\cdots$};
\node at (10.5, -9)  {$\cdots$};

\end{scriptsize}
\end{tikzpicture}
\captionsetup{aboveskip=0pt,font=it}
\caption{The graph $\ourE \times_1\Zb$ for $r=5, \underline{m}=(1,3)$.}
\end{figure}

\begin{figure}[H]
\begin{tikzpicture}[scale=0.8]
\begin{scriptsize}
\node at (-1.5,-2.5)  {Level 1};
\node at (-1.5,-6.5)  {Level 2};

\filldraw [black] (0,0) circle (1pt);
\filldraw [black] (2,-1) circle (1pt);
\filldraw [black] (4,-2) circle (1pt);
\filldraw [black] (6,-3) circle (1pt);
\filldraw [black] (8,-4) circle (1pt);
\filldraw [black] (10,0) circle (1pt);


\draw[cyan] [->] (0,0) -- (2,-1);
\draw[cyan] [->] (2,-1) -- (4,-2);
\draw[cyan] [->] (4,-2) -- (6,-3);
\draw[cyan] [->] (6,-3) -- (8,-4);
\draw[cyan] [->] (8,-4) -- (10,0);

\draw[] [dotted] (-1,-4.5) -- (11,-4.5);

\draw[orange] [->] (0,0) -- (2,-6);
\draw[orange] [->] (2,-1) -- (4,-7);
\draw[orange] [->] (4,-2) -- (6,-8);
\draw[orange] [->] (6,-3) -- (8,-9);
\draw[orange] [->] (8,-4) -- (10,-5);

\filldraw [black] (0,-5) circle (1pt);

\filldraw [black] (6,-5) circle (1pt);
\filldraw [black] (8,-5) circle (1pt);
\filldraw [black] (10,-5) circle (1pt);

\filldraw [black] (2,-6) circle (1pt);
\filldraw [black] (4,-6) circle (1pt);
\filldraw [black] (8,-6) circle (1pt);
\filldraw [black] (10,-6) circle (1pt);

\filldraw [black] (4,-7) circle (1pt);
\filldraw [black] (6,-7) circle (1pt);
\filldraw [black] (8,-7) circle (1pt);
\filldraw [black] (10,-7) circle (1pt);

\filldraw [black] (2,-8) circle (1pt);
\filldraw [black] (6,-8) circle (1pt);
\filldraw [black] (8,-8) circle (1pt);
\filldraw [black] (10,-8) circle (1pt);

\filldraw [black] (4,-9) circle (1pt);
\filldraw [black] (6,-9) circle (1pt);
\filldraw [black] (8,-9) circle (1pt);
\filldraw [black] (10,-9) circle (1pt);

\node at (-0.3, -5)  {$v_0^0$};
\node at (1.6, -5.7)  {$v_1^1$};
\node at (2, -8.3)  {$v_3^1$};
\node at (4, -5.7)  {$v_1^2$};
\node at (5.6, -5)  {$v_0^3$};

\draw[MidnightBlue] [->] (0,-5) -- (2,-8);
\draw[MidnightBlue] [->] (2,-8) -- (4,-6);
\draw[MidnightBlue] [->] (4,-6) -- (6,-9);
\draw[MidnightBlue] [->] (6,-9) -- (8,-7);
\draw[MidnightBlue] [->] (8,-7) -- (10,-5);

\draw[MidnightBlue] [->] (4,-7) -- (6,-5);
\draw[MidnightBlue] [->] (6,-5) -- (8,-8);
\draw[MidnightBlue] [->] (8,-8) -- (10,-6);

\draw[MidnightBlue] [->] (8,-9) -- (10,-7);

\draw[MidnightBlue] [->] (2,-6) -- (4,-9);
\draw[MidnightBlue] [->] (4,-9) -- (6,-7);
\draw[MidnightBlue] [->] (6,-7) -- (8,-5);
\draw[MidnightBlue] [->] (8,-5) -- (10,-8);

\draw[MidnightBlue] [->] (6,-8) -- (8,-6);
\draw[MidnightBlue] [->] (8,-6) -- (10,-9);

\node at (10.5, -0)  {$\cdots$};
\node at (10.5, -5)  {$\cdots$};
\node at (10.5, -6)  {$\cdots$};
\node at (10.5, -7)  {$\cdots$};
\node at (10.5, -8)  {$\cdots$};
\node at (10.5, -9)  {$\cdots$};

\end{scriptsize}
\end{tikzpicture}
\captionsetup{aboveskip=0pt,font=it}
\caption{The graph $F_{5;(1,3)}$.}
\label{F13}
\end{figure} 
\end{example}

We refer to the subgraphs given by the orbits of addition by $m_i$, depicted vertically in shades of blue above, as \emph{periodic lines}.

\begin{theorem}\label{isowF}
For all choices of $k$, $r$ and $\mm$ with $m_i\in\Zru$ we have 
\[
\QLS^\gamma\otimes\KKK \cong C^*(\ourF)\otimes\KKK
\]
and
\[
\QLS^\gamma\cong P^0C^*(\ourE \times_1\mathbb{Z})P^0
\]
where $P^0:=\sum_{v\in (L_{2k+1})^0}P_{((v,0),0)}$ in $C^*(\ourE \times_1\mathbb{Z})$.
\end{theorem}

\begin{proof}
We prove the last claim first. The fixed point algebra of \QLS\ by the circle action $\gamma$ inherited from \VS\ can by \eqref{thecorner} be described as the following corner of the fixed point algebra of $C^*(\ourE)$ under the gauge action.
$$
\QLS^{\gamma}\cong \sum_{v\in (L_{2k+1})^0} P_{(v,0)}C^*(\ourE)^{\gamma}\sum_{v\in (L_{2k+1})^0} P_{(v,0)}
$$
This follows since the circle action on \VS\ becomes the gauge action on $C^*(L_{2k+1})$ which corresponds to the gauge action on $C^*(\ourE)$ under the isomorphism in \eqref{thecorner} (see \cite[Section 2, equation (15),(16)]{jhhws:qlsga}). 

Due to \cite[Corollary 4.9]{tc:cga} we have 
$$
C^*(\ourE)^\gamma\cong \sum_{w\in (\ourE)^0} P_{(w,0)} C^*(\ourE \times_1\mathbb{Z}) \sum_{w\in E^0} P_{(w,0)}. 
$$
Under the above isomorphism we have $\sum_{v\in L_{2k+1}} P_{(v,0)}\mapsto \sum_{v\in L_{2k+1}} P_{((v,0),0)}$, hence
$$
\begin{aligned}
&C(L_q^{2k+1}(r;\underline{m}))^{\gamma}
\\
&\cong \sum_{v\in (L_{2k+1})^0} P_{((v,0),0)}\left(\sum_{w\in (\ourE)^0} P_{(w,0)} C^*(\ourE\times_1\mathbb{Z}) \sum_{w\in (\ourE)^0} P_{(w,0)}\right)\sum_{v\in (L_{2k+1})^0} P_{((v,0),0)} \\
&\cong \sum_{v\in (L_{2k+1})^0} P_{((v,0),0)}C^*(\ourE\times_1\mathbb{Z}) \sum_{v\in (L_{2k+1})^0} P_{((v,0),0)}
\end{aligned}
$$
Since $\QLS\cong \overline{\span}\{s_{\alpha}s_{\beta}^*: \alpha,\beta\in (\ourE\times_1\mathbb{Z})^*, r(\alpha)=r(\beta)\}$, an element $s_{\alpha}s_{\beta}^*$
can only be in the corner if $s(\alpha),s(\beta)\in \{((v,0),0): v\in (L_{2k+1})^0\}$. This implies that $r(\alpha)=r(\beta)\in (F_{r;\underline{m}})^0$ since $(F_{r;\underline{m}})^0$ by definition is the hereditary subset of $\ourE^0\times\Zb$ generated by the vertices $((v,0),0)$ for $v\in (L_{2k+1})^0$. Then 
\begin{equation}\label{corner1}
\begin{aligned}
\QLS^{\gamma}
&\cong \sum_{v\in (L_{2k+1})^0} P_{((v,0),0)}C^*(\ourE\times_1\mathbb{Z}) \sum_{v\in (L_{2k+1})^0} P_{((v,0),0)} \\
&\cong \sum_{v\in (L_{2k+1})^0} P_{((v,0),0)}C^*(\ourF) \sum_{v\in (L_{2k+1})^0} P_{((v,0),0)} 
\end{aligned}
\end{equation}
By construction of the graph $\ourF$, the projection $\sum_{v\in (L_{2k+1})^0} P_{((v,0),0)}$ is full in $C^*(\ourF)$. It follows that the corner \eqref{corner1} is stably isomorphic (or Morita equivalent) to  $C^*(\ourF)$ by \cite[Corollary 2.6]{lgb:sihsc}. 
\end{proof}

\begin{example}\label{exgraph2}
To compare with Example \ref{exgraph1} consider now $r=5$ and the weight vectors $\underline{m}=(2,3)$ and $\underline{m}=(2,2).$ The graphs $F_{5;(1,3)}$ and $F_{5;(2,3)}$ have in common that eventually they will reach all the vertices at level 2 after a particular point. The graph $F_{5;(2,2)}$ has a different structure since we only obtain one periodic line  in both level 1 and 2. We see immediately that $(C(L_q^{3}(5;(2,2))),\gamma)\not\simeq (C(L_q^{3}(5;(2,3))),\gamma)$.
\\
\begin{minipage}{8cm}
\begin{figure}[H]
\begin{tikzpicture}[scale=0.6]
\begin{scriptsize}

\filldraw [black] (0,0) circle (1pt);
\filldraw [black] (2,-2) circle (1pt);
\filldraw [black] (4,-4) circle (1pt);
\filldraw [black] (6,-1) circle (1pt);
\filldraw [black] (8,-3) circle (1pt);
\filldraw [black] (10,0) circle (1pt);

\draw[cyan] [->] (0,0) -- (2,-2);
\draw[cyan] [->] (2,-2) -- (4,-4);
\draw[cyan] [->] (4,-4) -- (6,-1);
\draw[cyan] [->] (6,-1) -- (8,-3);
\draw[cyan] [->] (8,-3) -- (10,0);

\draw[] [dotted] (-1,-4.5) -- (11,-4.5);

\draw[orange] [->] (0,0) -- (2,-7);
\draw[orange] [->] (2,-2) -- (4,-9);
\draw[orange] [->] (4,-4) -- (6,-6);
\draw[orange] [->] (6,-1) -- (8,-8);
\draw[orange] [->] (8,-3) -- (10,-5);

\filldraw [black] (0,-5) circle (1pt);
\filldraw [black] (4,-5) circle (1pt);
\filldraw [black] (8,-5) circle (1pt);
\filldraw [black] (10,-5) circle (1pt);

\filldraw [black] (4,-6) circle (1pt);
\filldraw [black] (6,-6) circle (1pt);
\filldraw [black] (8,-6) circle (1pt);
\filldraw [black] (10,-6) circle (1pt);

\filldraw [black] (2,-7) circle (1pt);
\filldraw [black] (6,-7) circle (1pt);
\filldraw [black] (8,-7) circle (1pt);
\filldraw [black] (10,-7) circle (1pt);

\filldraw [black] (2,-8) circle (1pt);
\filldraw [black] (6,-8) circle (1pt);
\filldraw [black] (8,-8) circle (1pt);
\filldraw [black] (10,-8) circle (1pt);

\filldraw [black] (4,-9) circle (1pt);
\filldraw [black] (6,-9) circle (1pt);
\filldraw [black] (8,-9) circle (1pt);
\filldraw [black] (10,-9) circle (1pt);

\draw[MidnightBlue] [->] (0,-5) -- (2,-8);
\draw[MidnightBlue] [->] (2,-8) -- (4,-6);
\draw[MidnightBlue] [->] (4,-6) -- (6,-9);
\draw[MidnightBlue] [->] (6,-9) -- (8,-7);
\draw[MidnightBlue] [->] (8,-7) -- (10,-5);

\draw[MidnightBlue] [->] (8,-8) -- (10,-6);

\draw[MidnightBlue] [->] (6,-6) -- (8,-9);
\draw[MidnightBlue] [->] (8,-9) -- (10,-7);

\draw[MidnightBlue] [->] (4,-9) -- (6,-7);
\draw[MidnightBlue] [->] (6,-7) -- (8,-5);
\draw[MidnightBlue] [->] (8,-5) -- (10,-8);

\draw[MidnightBlue] [->] (2,-7) -- (4,-5);
\draw[MidnightBlue] [->] (4,-5) -- (6,-8);
\draw[MidnightBlue] [->] (6,-8) -- (8,-6);
\draw[MidnightBlue] [->] (8,-6) -- (10,-9);

\node at (10.5, -0)  {$\cdots$};
\node at (10.5, -5)  {$\cdots$};
\node at (10.5, -6)  {$\cdots$};
\node at (10.5, -7)  {$\cdots$};
\node at (10.5, -8)  {$\cdots$};
\node at (10.5, -9)  {$\cdots$};

\end{scriptsize}
\end{tikzpicture}
\captionsetup{aboveskip=0pt,font=it}
\caption{The graph $F_{5;(2,3)}$.}
\label{F23}
\end{figure}
\end{minipage}
\begin{minipage}{8cm}
\begin{figure}[H]
\begin{tikzpicture}[scale=0.6]
\begin{scriptsize}

\filldraw [black] (0,0) circle (1pt);
\filldraw [black] (2,-2) circle (1pt);
\filldraw [black] (4,-4) circle (1pt);
\filldraw [black] (6,-1) circle (1pt);
\filldraw [black] (8,-3) circle (1pt);
\filldraw [black] (10,0) circle (1pt);

\draw[cyan] [->] (0,0) -- (2,-2);
\draw[cyan] [->] (2,-2) -- (4,-4);
\draw[cyan] [->] (4,-4) -- (6,-1);
\draw[cyan] [->] (6,-1) -- (8,-3);
\draw[cyan] [->] (8,-3) -- (10,0);

\draw[] [dotted] (-1,-4.5) -- (11,-4.5);

\draw[orange] [->] (0,0) -- (2,-7);
\draw[orange] [->] (2,-2) -- (4,-9);
\draw[orange] [->] (4,-4) -- (6,-6);
\draw[orange] [->] (6,-1) -- (8,-8);
\draw[orange] [->] (8,-3) -- (10,-5);

\filldraw [black] (0,-5) circle (1pt);
\filldraw [black] (2,-7) circle (1pt);
\filldraw [black] (4,-9) circle (1pt);
\filldraw [black] (6,-6) circle (1pt);
\filldraw [black] (8,-8) circle (1pt);
\filldraw [black] (10,-5) circle (1pt);

\draw[MidnightBlue] [->] (0,-5) -- (2,-7);
\draw[MidnightBlue] [->] (2,-7) -- (4,-9);
\draw[MidnightBlue] [->] (4,-9) -- (6,-6);
\draw[MidnightBlue] [->] (6,-6) -- (8,-8);
\draw[MidnightBlue] [->] (8,-8) -- (10,-5);

\node at (10.5, -0)  {$\cdots$};
\node at (10.5, -5)  {$\cdots$};
\end{scriptsize}
\end{tikzpicture}
\captionsetup{aboveskip=0pt,font=it}
\caption{The graph $F_{5;(2,2)}$.}
\label{F22}
\end{figure}
\end{minipage}
\end{example}

\begin{example}\label{eight}
For $r=8$ the graph $F_{r,\underline{m}}$ takes the form as in figure \ref{F8} where we either obtain 1, 2 or 4 periodic lines in level 2. Note that the number of periodic lines in level 2 is precisely $r/\gcd(m_2-m_1,r)$ which we in Proposition \ref{PropSetOfNumbers} will see is true in general . 
\begin{figure}[H]
\hspace{-1.8cm}\begin{minipage}{4cm}
\begin{tikzpicture}[scale=0.6]
\begin{scriptsize}
\node at (4,0) {$F_{8;(1,1)}$};
\filldraw [black] (0,0) circle (1pt);
\filldraw [black] (1,-1) circle (1pt);
\filldraw [black] (2,-2) circle (1pt);
\filldraw [black] (3,-3) circle (1pt);
\filldraw [black] (4,-4) circle (1pt);
\filldraw [black] (5,-5) circle (1pt);
\filldraw [black] (6,-6) circle (1pt);
\filldraw [black] (7,-7) circle (1pt);

\filldraw [black] (8,0) circle (1pt);
\node at (8.5, 0)  {$\cdots$};
\draw[cyan] [->] (7,-7) -- (8,0);

\draw[cyan] [->] (0,0) -- (1,-1);
\draw[cyan] [->] (1,-1) -- (2,-2);
\draw[cyan] [->] (2,-2) -- (3,-3);
\draw[cyan] [->] (3,-3) -- (4,-4);
\draw[cyan] [->] (4,-4) -- (5,-5);
\draw[cyan] [->] (5,-5) -- (6,-6);
\draw[cyan] [->] (6,-6) -- (7,-7);

\draw[] [dotted] (-0.5,-7.5) -- (8.5,-7.5);

\draw[orange] [->] (0,0) -- (1,-9);
\draw[orange] [->] (1,-1) -- (2,-10);
\draw[orange] [->] (2,-2) -- (3,-11);
\draw[orange] [->] (3,-3) -- (4,-12);
\draw[orange] [->] (4,-4) -- (5,-13);
\draw[orange] [->] (5,-5) -- (6,-14);
\draw[orange] [->] (6,-6) -- (7,-15);
\draw[orange] [->] (7,-7) -- (8,-8);

\filldraw [black] (0,-8) circle (1pt);
\filldraw [black] (1,-9) circle (1pt);
\filldraw [black] (2,-10) circle (1pt);
\filldraw [black] (3,-11) circle (1pt);
\filldraw [black] (4,-12) circle (1pt);
\filldraw [black] (5,-13) circle (1pt);
\filldraw [black] (6,-14) circle (1pt);
\filldraw [black] (7,-15) circle (1pt);

\draw[MidnightBlue] [->] (0,-8) -- (1,-9);
\draw[MidnightBlue] [->] (1,-9) -- (2,-10);
\draw[MidnightBlue] [->] (2,-10) -- (3,-11);
\draw[MidnightBlue] [->] (3,-11) -- (4,-12);
\draw[MidnightBlue] [->] (4,-12) -- (5,-13);
\draw[MidnightBlue] [->] (5,-13) -- (6,-14);
\draw[MidnightBlue] [->] (6,-14) -- (7,-15);
\draw[MidnightBlue] [->] (7,-15) -- (8,-8);

\node at (8.5, -8)  {$\cdots$};

\end{scriptsize}
\end{tikzpicture}
\end{minipage}
\hspace{1.5cm}
\begin{minipage}{3.5cm}
\begin{tikzpicture}[scale=0.6]
\begin{scriptsize}
\node at (4,0) {$F_{8;(1,5)}$};
\filldraw [black] (0,0) circle (1pt);
\filldraw [black] (1,-1) circle (1pt);
\filldraw [black] (2,-2) circle (1pt);
\filldraw [black] (3,-3) circle (1pt);
\filldraw [black] (4,-4) circle (1pt);
\filldraw [black] (5,-5) circle (1pt);
\filldraw [black] (6,-6) circle (1pt);
\filldraw [black] (7,-7) circle (1pt);

\filldraw [black] (8,0) circle (1pt);
\node at (8.5, 0)  {$\cdots$};
\draw[cyan] [->] (7,-7) -- (8,0);

\draw[cyan] [->] (0,0) -- (1,-1);
\draw[cyan] [->] (1,-1) -- (2,-2);
\draw[cyan] [->] (2,-2) -- (3,-3);
\draw[cyan] [->] (3,-3) -- (4,-4);
\draw[cyan] [->] (4,-4) -- (5,-5);
\draw[cyan] [->] (5,-5) -- (6,-6);
\draw[cyan] [->] (6,-6) -- (7,-7);

\draw[] [dotted] (-0.5,-7.5) -- (8.5,-7.5);

\draw[orange] [->] (0,0) -- (1,-9);
\draw[orange] [->] (1,-1) -- (2,-10);
\draw[orange] [->] (2,-2) -- (3,-11);
\draw[orange] [->] (3,-3) -- (4,-12);
\draw[orange] [->] (4,-4) -- (5,-13);
\draw[orange] [->] (5,-5) -- (6,-14);
\draw[orange] [->] (6,-6) -- (7,-15);
\draw[orange] [->] (7,-7) -- (8,-8);

\filldraw [black] (0,-8) circle (1pt);
\filldraw [black] (4,-8) circle (1pt);
\filldraw [black] (8,-8) circle (1pt);

\filldraw [black] (1,-9) circle (1pt);
\filldraw [black] (5,-9) circle (1pt);

\filldraw [black] (2,-10) circle (1pt);
\filldraw [black] (6,-10) circle (1pt);

\filldraw [black] (3,-11) circle (1pt);
\filldraw [black] (7,-11) circle (1pt);

\filldraw [black] (4,-12) circle (1pt);
\filldraw [black] (8,-12) circle (1pt);

\filldraw [black] (1,-13) circle (1pt);
\filldraw [black] (5,-13) circle (1pt);

\filldraw [black] (2,-14) circle (1pt);
\filldraw [black] (6,-14) circle (1pt);

\filldraw [black] (3,-15) circle (1pt);
\filldraw [black] (7,-15) circle (1pt);

\draw[MidnightBlue] [->] (0,-8) -- (1,-13);
\draw[MidnightBlue] [->] (1,-13) -- (2,-10);
\draw[MidnightBlue] [->] (2,-10) -- (3,-15);
\draw[MidnightBlue] [->] (3,-15) -- (4,-12);
\draw[MidnightBlue] [->] (4,-12) -- (5,-9);
\draw[MidnightBlue] [->] (5,-9) -- (6,-14);
\draw[MidnightBlue] [->] (6,-14) -- (7,-11);
\draw[MidnightBlue] [->] (7,-11) -- (8,-8);

\draw[MidnightBlue] [->] (1,-9) -- (2,-14);
\draw[MidnightBlue] [->] (2,-14) -- (3,-11);
\draw[MidnightBlue] [->] (3,-11) -- (4,-8);
\draw[MidnightBlue] [->] (4,-8) -- (5,-13);
\draw[MidnightBlue] [->] (5,-13) -- (6,-10);
\draw[MidnightBlue] [->] (6,-10) -- (7,-15);
\draw[MidnightBlue] [->] (7,-15) -- (8,-12);

\node at (8.5, -8)  {$\cdots$};
\node at (8.5, -12)  {$\cdots$};
\end{scriptsize}
\end{tikzpicture}
\end{minipage}
\hspace{2cm}
\begin{minipage}{3.5cm}
\begin{tikzpicture}[scale=0.6]
\begin{scriptsize}
\node at (4,0) {$F_{8;(1,3)}$};
\filldraw [black] (0,0) circle (1pt);
\filldraw [black] (1,-1) circle (1pt);
\filldraw [black] (2,-2) circle (1pt);
\filldraw [black] (3,-3) circle (1pt);
\filldraw [black] (4,-4) circle (1pt);
\filldraw [black] (5,-5) circle (1pt);
\filldraw [black] (6,-6) circle (1pt);
\filldraw [black] (7,-7) circle (1pt);

\filldraw [black] (8,0) circle (1pt);
\node at (8.5, 0)  {$\cdots$};
\draw[cyan] [->] (7,-7) -- (8,0);

\draw[cyan] [->] (0,0) -- (1,-1);
\draw[cyan] [->] (1,-1) -- (2,-2);
\draw[cyan] [->] (2,-2) -- (3,-3);
\draw[cyan] [->] (3,-3) -- (4,-4);
\draw[cyan] [->] (4,-4) -- (5,-5);
\draw[cyan] [->] (5,-5) -- (6,-6);
\draw[cyan] [->] (6,-6) -- (7,-7);

\draw[] [dotted] (-0.5,-7.5) -- (8.5,-7.5);

\draw[orange] [->] (0,0) -- (1,-9);
\draw[orange] [->] (1,-1) -- (2,-10);
\draw[orange] [->] (2,-2) -- (3,-11);
\draw[orange] [->] (3,-3) -- (4,-12);
\draw[orange] [->] (4,-4) -- (5,-13);
\draw[orange] [->] (5,-5) -- (6,-14);
\draw[orange] [->] (6,-6) -- (7,-15);
\draw[orange] [->] (7,-7) -- (8,-8);

\filldraw [black] (0,-8) circle (1pt);
\filldraw [black] (4,-8) circle (1pt);
\filldraw [black] (6,-8) circle (1pt);
\filldraw [black] (8,-8) circle (1pt);

\filldraw [black] (1,-9) circle (1pt);
\filldraw [black] (3,-9) circle (1pt);
\filldraw [black] (5,-9) circle (1pt);
\filldraw [black] (7,-9) circle (1pt);

\filldraw [black] (2,-10) circle (1pt);
\filldraw [black] (4,-10) circle (1pt);
\filldraw [black] (6,-10) circle (1pt);
\filldraw [black] (8,-10) circle (1pt);

\filldraw [black] (1,-11) circle (1pt);
\filldraw [black] (3,-11) circle (1pt);
\filldraw [black] (5,-11) circle (1pt);
\filldraw [black] (7,-11) circle (1pt);

\filldraw [black] (2,-12) circle (1pt);
\filldraw [black] (4,-12) circle (1pt);
\filldraw [black] (6,-12) circle (1pt);
\filldraw [black] (8,-12) circle (1pt);

\filldraw [black] (3,-13) circle (1pt);
\filldraw [black] (5,-13) circle (1pt);
\filldraw [black] (7,-13) circle (1pt);

\filldraw [black] (2,-14) circle (1pt);
\filldraw [black] (4,-14) circle (1pt);
\filldraw [black] (6,-14) circle (1pt);
\filldraw [black] (8,-14) circle (1pt);

\filldraw [black] (3,-15) circle (1pt);
\filldraw [black] (5,-15) circle (1pt);
\filldraw [black] (7,-15) circle (1pt);

\draw[MidnightBlue] [->] (0,-8) -- (1,-11);
\draw[MidnightBlue] [->] (1,-11) -- (2,-14);
\draw[MidnightBlue] [->] (2,-14) -- (3,-9);
\draw[MidnightBlue] [->] (3,-9) -- (4,-12);
\draw[MidnightBlue] [->] (4,-12) -- (5,-15);
\draw[MidnightBlue] [->] (5,-15) -- (6,-10);
\draw[MidnightBlue] [->] (6,-10) -- (7,-13);
\draw[MidnightBlue] [->] (7,-13) -- (8,-8);

\draw[MidnightBlue] [->] (1,-9) -- (2,-12);
\draw[MidnightBlue] [->] (2,-12) -- (3,-15);
\draw[MidnightBlue] [->] (3,-15) -- (4,-10);
\draw[MidnightBlue] [->] (4,-10) -- (5,-13);
\draw[MidnightBlue] [->] (5,-13) -- (6,-8);
\draw[MidnightBlue] [->] (6,-8) -- (7,-11);
\draw[MidnightBlue] [->] (7,-11) -- (8,-14);

\draw[MidnightBlue] [->] (2,-10) -- (3,-13);
\draw[MidnightBlue] [->] (3,-13) -- (4,-8);
\draw[MidnightBlue] [->] (4,-8) -- (5,-11);
\draw[MidnightBlue] [->] (5,-11) -- (6,-14);
\draw[MidnightBlue] [->] (6,-14) -- (7,-9);
\draw[MidnightBlue] [->] (7,-9) -- (8,-12);

\draw[MidnightBlue] [->] (3,-11) -- (4,-14);
\draw[MidnightBlue] [->] (4,-14) -- (5,-9);
\draw[MidnightBlue] [->] (5,-9) -- (6,-12);
\draw[MidnightBlue] [->] (6,-12) -- (7,-15);
\draw[MidnightBlue] [->] (7,-15) -- (8,-10);

\node at (8.5, -8)  {$\cdots$};
\node at (8.5, -10)  {$\cdots$};
\node at (8.5, -12)  {$\cdots$};
\node at (8.5, -14)  {$\cdots$};
\end{scriptsize}
\end{tikzpicture}
\end{minipage}
\captionsetup{aboveskip=0pt,font=it}
\caption{The graphs $F_{8;(m_1,m_2)}$.}
\label{F8}
\end{figure}
\end{example}

\begin{example}\label{dim5}
In dimension $5$ the graph $\ourF$ has $3$ levels. In Figure \ref{Fdim5} we indicate two examples of the structure of $\ourF$ in the case where $r=3$.
\begin{figure}[H]
\begin{minipage}{4cm}
\begin{tikzpicture}[scale=0.6]
\begin{scriptsize}
\node at (2,-9) {$F_{3;(1,1,2)}$};
\filldraw [black] (0,0) circle (1pt);
\filldraw [black] (1,-1) circle (1pt);
\filldraw [black] (2,-2) circle (1pt);
\filldraw [black] (3,-3) circle (1pt);
\filldraw [black] (3,0) circle (1pt);

\draw[cyan] [->] (0,0) -- (1,-1);
\draw[cyan] [->] (1,-1) -- (2,-2);
\draw[cyan] [->] (2,-2) -- (3,0);
\draw[orange] [->] (2,-2) -- (3,-3);

\draw[] [dotted] (-0.5,-2.5) -- (4.5,-2.5);

\draw[orange] [->] (0,0) -- (1,-4);
\draw[orange] [->] (1,-1) -- (2,-5);

\filldraw [black] (0,-3) circle (1pt);
\filldraw [black] (1,-4) circle (1pt);
\filldraw [black] (2,-5) circle (1pt);

\draw[MidnightBlue] [->] (0,-3) -- (1,-4);
\draw[MidnightBlue] [->] (1,-4) -- (2,-5);
\draw[MidnightBlue] [->] (2,-5) -- (3,-3);

\draw[] [dotted] (-0.5,-5.5) -- (4.5,-5.5);

\filldraw [black] (0,-6) circle (1pt);
\filldraw [black] (1,-8) circle (1pt);
\filldraw [black] (2,-7) circle (1pt);
\filldraw [black] (3,-6) circle (1pt);

\filldraw [black] (1,-7) circle (1pt);
\filldraw [black] (2,-8) circle (1pt);

\filldraw [black] (2,-6) circle (1pt);
\filldraw [black] (3,-7) circle (1pt);
\filldraw [black] (3,-8) circle (1pt);

\draw[blue] [->] (0,-6) -- (1,-8);
\draw[blue] [->] (1,-8) -- (2,-7);
\draw[blue] [->] (2,-7) -- (3,-6);
\draw[blue] [->] (1,-7) -- (2,-6);
\draw[blue] [->] (2,-8) -- (3,-7);
\draw[blue] [->] (2,-6) -- (3,-8);

\draw[Goldenrod] [->] (0,0) -- (1,-7);
\draw[Goldenrod] [->] (1,-1) -- (2,-8);
\draw[Goldenrod] [->] (2,-2) -- (3,-6);

\draw[orange] [->] (0,-3) -- (1,-7);
\draw[orange] [->] (1,-4) -- (2,-8);
\draw[orange] [->] (2,-5) -- (3,-6);

\node at (3.5, -0)  {$\cdots$};
\node at (3.5, -3)  {$\cdots$};
\node at (3.5, -6)  {$\cdots$};
\node at (3.5, -7)  {$\cdots$};
\node at (3.5, -8)  {$\cdots$};

\end{scriptsize}
\end{tikzpicture}
\end{minipage}
\ \ \ 
\begin{minipage}{4cm}
\begin{tikzpicture}[scale=0.6]
\begin{scriptsize}
\node at (2,-9) {$F_{3;(1,2,1)}$};
\filldraw [black] (0,0) circle (1pt);
\filldraw [black] (1,-1) circle (1pt);
\filldraw [black] (2,-2) circle (1pt);
\filldraw [black] (3,-3) circle (1pt);
\filldraw [black] (3,0) circle (1pt);

\draw[cyan] [->] (0,0) -- (1,-1);
\draw[cyan] [->] (1,-1) -- (2,-2);
\draw[cyan] [->] (2,-2) -- (3,0);
\draw[orange] [->] (2,-2) -- (3,-3);

\draw[] [dotted] (-0.5,-2.5) -- (4.5,-2.5);

\draw[orange] [->] (0,0) -- (1,-4);
\draw[orange] [->] (1,-1) -- (2,-5);

\draw[] [dotted] (-0.5,-5.5) -- (4.5,-5.5);

\filldraw [black] (0,-3) circle (1pt);
\filldraw [black] (1,-5) circle (1pt);
\filldraw [black] (2,-4) circle (1pt);
\filldraw [black] (3,-3) circle (1pt);

\filldraw [black] (1,-4) circle (1pt);
\filldraw [black] (2,-5) circle (1pt);

\filldraw [black] (2,-3) circle (1pt);
\filldraw [black] (3,-4) circle (1pt);
\filldraw [black] (3,-5) circle (1pt);

\draw[MidnightBlue] [->] (0,-3) -- (1,-5);
\draw[MidnightBlue] [->] (1,-5) -- (2,-4);
\draw[MidnightBlue] [->] (2,-4) -- (3,-3);
\draw[MidnightBlue] [->] (1,-4) -- (2,-3);
\draw[MidnightBlue] [->] (2,-5) -- (3,-4);
\draw[MidnightBlue] [->] (2,-3) -- (3,-5);

\filldraw [black] (0,-6) circle (1pt);
\filldraw [black] (1,-7) circle (1pt);
\filldraw [black] (2,-8) circle (1pt);
\filldraw [black] (3,-6) circle (1pt);

\filldraw [black] (1,-8) circle (1pt);

\filldraw [black] (2,-6) circle (1pt);
\filldraw [black] (2,-7) circle (1pt);
\filldraw [black] (3,-7) circle (1pt);
\filldraw [black] (3,-8) circle (1pt);

\draw[blue] [->] (0,-6) -- (1,-7);
\draw[blue] [->] (1,-7) -- (2,-8);
\draw[blue] [->] (2,-8) -- (3,-6);

\draw[blue] [->] (1,-8) -- (2,-6);
\draw[blue] [->] (2,-6) -- (3,-7);
\draw[blue] [->] (2,-7) -- (3,-8);

\draw[Goldenrod] [->] (0,0) -- (1,-7);
\draw[Goldenrod] [->] (1,-1) -- (2,-8);
\draw[Goldenrod] [->] (2,-2) -- (3,-6);

\draw[orange] [->] (0,-3) -- (1,-8);
\draw[orange] [->] (1,-5) -- (2,-7);
\draw[orange] [->] (2,-4) -- (3,-6);
\draw[orange] [->] (1,-4) -- (2,-6);
\draw[orange] [->] (2,-5) -- (3,-7);
\draw[orange] [->] (2,-3) -- (3,-8);

\node at (3.5, -0)  {$\cdots$};
\node at (3.5, -3)  {$\cdots$};
\node at (3.5, -6)  {$\cdots$};
\node at (3.5, -7)  {$\cdots$};
\node at (3.5, -8)  {$\cdots$};

\end{scriptsize}
\end{tikzpicture}
\end{minipage}
\captionsetup{aboveskip=0pt,font=it}
\caption{Some graphs $F_{3;(m_1,m_2,m_3)}$.}
\label{Fdim5}
\end{figure}
\end{example}

\begin{remark}
We can read off the paths $\alpha_{t_1,t_2}:=e_{11}^{t_1}e_{12}e_{22}^{t_2}$ that lie inside $\mathcal{A}(r;(m_1,m_2))$ on the graph $F_{r,\underline{m}}$. The path $\alpha_{t_1,t_2}$ corresponds to the path in $F_{r,\underline{m}}$ (see Figure \ref{F13}) which has source $w_0$ and then follows the blue edges in level 1 $t_1$-times before it goes to the second level by the orange edge. At level 2 the path follows the blue edges until it reaches the vertex $v_{0}^i$ for the smallest possible $i$. Then the number $t_2$ is precisely the number of blue edges in level 2 which the path has to follow to reach the first $v_{0}^i$. This corresponds precisely to the so called admissible paths in the skew product graph used in \cite{jhhws:qlsga}.
\end{remark}


We now formalize the structure observed in these examples. 
We first define $\PP(r;\mm)$ as the set
\[
 \{1,\dots, k+1\}\times  \{0,\dots, r-1\} 
\]
equipped with the order
\[
(i,j)\myto(i',j')
\]
defined as the reflexive and transitive closure of
\[
(i,j)\myo (i',j')\Longleftrightarrow \left\{\begin{array}{c}i<i'\\\exists n\in \Zr: n(m_{i}-m_{i'})\equiv j-j' \Mod r\end{array}\right\}
\]
We let $\PP_0(r;\mm)$ denote the elements dominated by $(i,0)$ for some $i$.

\begin{lemma}\label{interpretmyo}
The statements
\begin{enumerate}[(i)]
\item $(i,j)\myo (i',j')$ in $\PP(r;\mm)$
\item There is a path from the periodic line in $\ourE\times_1\ZZ$ containing $((v_i,j),0)$ to the periodic line containing $((v_i',j'),0)$, not visiting another periodic line along the way
\item $i'>i$ and $\gcd(m_{i'}-m_i,r)\mid j'-j$
\end{enumerate}
are equivalent.
\end{lemma}

\begin{proof} 
The relation $\myo$ precisely captures the situation where there is a direct path as indicated, and the condition defining $\myo$ is clearly equivalent to the one given with $\gcd$.
\end{proof}

\begin{lemma}\label{simplehasse}
When $(i,j)\myto(i',j')$, there exist $j_1,\dots j_{k-1}$ with  $k=i-i'$  so that 
\[
(i,j)\myo(i+1,j_1)\myo\cdots\myo (i+k-1,j_{k-1})\myo (i',j')
\]
\end{lemma}
\begin{proof}
When $(i,j)\myto(i',j')$ we have
\[
(i,j)=(i_0,j_0)\myo(i_1,j_1)\myo\cdots\myo (i_\ell,j_{\ell})= (i',j')
\]
where necessarily the $i_k$ are increasing. If $i_1>i_0+1$ we note that we know from the outset that  $n(m_{i_0}-m_{i_1})\equiv j_0-j_1 \mod r$, so with 
\[
j''=j_0-n(m_{i_0}-m_{i_0+1})
\]
we obtain
\[
n(m_{i_0}-m_{{i_0}+1})\equiv j_0-j''\qquad n(m_{i_0+1}-m_{i_1})\equiv j''-j_0 \mod r
\]
which allows us to extend the chain to
\[
(i_0,j_0)\myo(i_0+1,j'')\myo(i_1,j_1).
\]
Repeating the argument leads to the general claim.
\end{proof}


We recall the concept of \emph{Hasse diagrams} to the reader; graphs describing finite partially ordered sets as undirected graphs  so that when a line connects two vertices, these vertices are never on the same vertical level and indicating that the higher one dominates the lower. Further, Hasse diagrams never indicate indirect relations that may be inferred from the transitivity of partial orderings.
\\
Consequently, Lemma \ref{simplehasse} says that the Hasse diagram of $(\PP(r;\mm),\myto)$ and $(\PP_0(r;\mm),\myto)$ will always be stratified in the sense that all edges will go between the layers with vertices $(i,\bullet)$ to layers with vertices $(i+1,\bullet)$.

\begin{example}
The three different ideal structures exhibited in Example \ref{eight} give $\PP_0(8;\mm)$ as indicated by Hasse diagrams
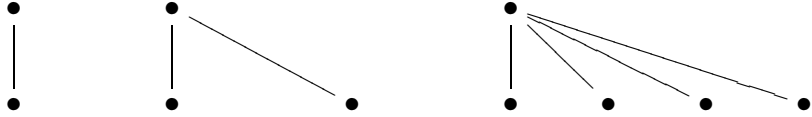
\[
\xymatrix{\bullet\ar@{-}[d]\\\bullet}\qquad\qquad
\xymatrix{\bullet\ar@{-}[d]\ar@{-}[drr]\\\bullet&&\bullet}\qquad\qquad
\xymatrix{\bullet\ar@{-}[d]\ar@{-}[dr]\ar@{-}[drr]\ar@{-}[drrr]\\\bullet&\bullet&\bullet&\bullet}
\]
\end{example}

\begin{example}
The two different ideal structures exhibited in Example \ref{dim5} give $\PP_0(3;\mm)$ as indicated by Hasse diagrams
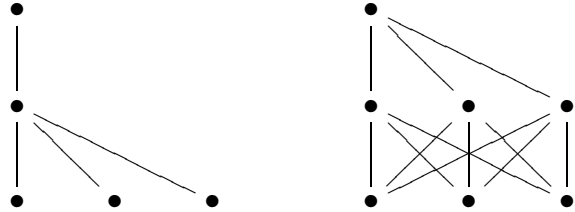
\[
\xymatrix{\bullet\ar@{-}[d]\\\bullet\ar@{-}[d]\ar@{-}[dr]\ar@{-}[drr]\\\bullet&\bullet&\bullet}\qquad\qquad
\xymatrix{\bullet\ar@{-}[d]\ar@{-}[dr]\ar@{-}[drr]\\\bullet\ar@{-}[d]\ar@{-}[dr]\ar@{-}[drr]&\bullet\ar@{-}[dl]\ar@{-}[d]\ar@{-}[dr]&\ar@{-}[d]\ar@{-}[dl]\ar@{-}[dll]\bullet\\\bullet&\bullet&\bullet}
\]
\end{example}

\begin{theorem}\label{Thm:IdealLattice} The following are equivalent for fixed $r$ and $k$ and varying $m_i,n_i\in\Zru$:
\begin{enumerate}[(i)]
\item 
The lattice of ideals of $\QLS^{\gamma}$ and $\QLSalt^{\gamma}$ are isomorphic
\item The lattice of ideals of $C^*(\ourE\times_1\ZZ)$ and $C^*(\ourEalt\times_1\ZZ)$ are isomorphic.
\item $(\PP_0(r;\mm),\myto)= (\PP_0(r;\nn),\myto)$  as partially ordered sets
\item $(\PP(r;\mm),\myto)= (\PP(r;\nn),\myto)$  as partially ordered sets
\item For all $i=1,\dots,k$ we have $\gcd(m_{i+1}-m_i,r)=\gcd(n_{i+1}-n_i,r)$
\end{enumerate}
\end{theorem}

It is obvious by (i) that all quantities discussed in (i)--(v) are invariants of equivariant isomorphism of the \QLS. As we shall discuss further below, they are not in general complete.

\begin{proof}
We start out by proving that (ii) and (iv) are equivalent.

To show this, we note that since the $C^*$-algebras in (i) and (ii) are graph $C^*$-algebras of real rank zero (even AF), their ideal lattices are again isomorphic to the lattice of hereditary and saturated sets of the graphs $\ourE\times_1\ZZ$ and $\ourEalt\times_1\ZZ$. Each hereditary set can be described uniquely by the periodic lines it contains, and since each such line contains exactly one vertex $((v,i),0)$, the hereditary subset is uniquely determined by the 
 subset of $\PP(r;\mm)$ thus defined.

As noted in Lemma \ref{interpretmyo}(ii), we  have defined $\myo$ reflecting the direct arrows from one circle to another, so $(i,j)\myto(i',j')$ precisely captures the situation that whenever the periodic line at $(i,j)$ is contained in a hereditary set, so must the line at $(i',j')$. It is easy to see that any hereditary set must also be saturated, so we conclude that (i) is equivalent to the claim that the lattices of  downward directed subsets of  $(\PP(r;\mm),\myto)$ and $(\PP(r;\nn),\myto)$ are isomorphic. It is a standard fact (explained in \cite{memgap:cctdl}) that this is equivalent to (iii).

The argument that (i) and (iii) are equivalent follows in exactly the same way, using also that by Theorem \ref{isowF} we see that the ideal lattice of $\QLS^\gamma$ is isomorphic to that of $C^*(F_{r;\mm})$. 

When (v) holds, we get by invoking Lemma \ref{interpretmyo}(iii) that $(i,j)\myo (i+1,j')$ simultaneously in the four  ordered sets considered in (iii) and (iv), and by Lemma \ref{simplehasse} this shows that these sets are pairwise isomorphic (in fact identical). Conversely, by considering a maximal element in either set and working downwards allows us to read off $r/\gcd(m_{i'}-m_i,r)$ as the number of immediate predecessors at each level.
\end{proof}

We denote by $\divisors(n)$ the number of positive divisors of $n$. The following result is elementary, but we were not able to locate it in the literature.

\begin{lemma}
\[
\{d: d\mid r, 2\mid r\Longrightarrow 2\mid d\}=\{\gcd(x-y,r): x,y\in \Zru\}
\]
and consequently the number of elements in the sets are $\divisors(r)$ for odd $r$ and $\divisors(r/2)$ for even $r$.
\end{lemma}
\begin{proof}
For the inclusion from right to left,  fix $z=\operatorname{gcd}(x-y, r)$. Obviously, $z\mid r$, and if $r$ is even, both $x$ and $y$ must be odd to be prime to it. Hence $2\mid z$.

In the other direction, write
$$
r=p_{1}^{n_{1}} p_{2}^{n_{2}} \cdots p_{k}^{n_{k}}
$$
where $n_{i} \geq 1$ and $p_{1}<p_{2}<\cdots<p_{n}$.
For $d$ in the leftmost set, we write
$$
d=p_{1}^{m_{1}} p_{2}^{m_{2}} \cdots p_{k}^{m_{k}}
$$
where $0 \leq m_{i} \leq n_{i}$. Note that if $p_{1}=2$,
then $m_{1}>0$.

We now appeal to the Chinese remainder theorem, noting that the canonical map
$$
\varphi: \mathbb{Z} / r \rightarrow \mathbb{Z} / p_{1}^{n_{1}} \times \cdots \times \mathbb{Z} / p_{k}^{n_{k}}
$$
given by
$$
\varphi(\ell)=\left(x \bmod p_{1}^{n_{1}}, \ldots, x \bmod p_{k}^{n_{k}}\right)
$$
is a ring isomorphism, and that
$$
\gcd(x, r)=p_{1}^{\ell_{1}} \cdots p_{k}^{\ell_{k}}
$$
where 
$$
\ell_{i}=\max \left\{\ell \leq n_{i}:p_{i}^{\ell}\mid\left(x \bmod p_{i}^{m_{i}}\right)\right.\}.
$$
The latter observation follows from noting that $p_{i}$ cannot divide $\operatorname{gcd}(x, r)$ more than $n_i$ times because of $r$, and since
$$
x-(x \bmod p_{i}^{n_i})=y p_{i}^{n_i}
$$
for some $y$, we see that $p_{i}^{\ell}$ for $\ell \leq n_{i}$ divides $x$ if and only if it divides $x \bmod p_{i}^{n_i}$.

Now we construct $x$ so that
$$
\begin{aligned}
\operatorname{gcd}(x, r)=1 \ \text{and} \ \operatorname{gcd}(x-1, r)=d.
\end{aligned}
$$
Let
$$
\phi(d)=\left(\delta_{1}, \ldots, \delta_{k}\right)
$$
and let $x=\varphi^{-1}\left(\xi_{1}, \dots, \xi_{k}\right)$ with
$$
\xi_{i}= \begin{cases}\delta_{i}+1 & p_{i}\nmid \delta_{i}+1 \\ \delta_{i} & p_{i} \mid \delta_{i}+1\end{cases}
$$
Note that $p_{i} \mid \xi_{i}$ in $\mathbb{Z}/p_{i}^{n_{i}}$ by construction, so that $\operatorname{gcd} ( x, r)=1$ as seen above.

Let $z=x-1$ and note that since
$\varphi$ is a ring isomorphism,
$$
\varphi(z)=\left(\zeta_{1}, \ldots, \zeta_{n}\right)
$$
with
$$
\zeta_{i}= \begin{cases}\delta_{i} & p_{i} \nmid \delta_{i}+1 \\ \delta_{i}-1 & p_{i} \mid \delta_{i}+1\end{cases}
$$

If $m_{i}>0$, we know that $p_{i} \mid \delta_{i}$ and consequently $p_{i} \nmid\delta_{i}+1$, and so $\zeta_{i}=\delta_{i}$. We conclude that $p_{i}^{m_i} \mid \zeta_{i}$ as a maximal power.
If $m_{i}=0$, we have noted that $p_{i}>2$. If $\beta_{i} \mid \delta_{i}+1$, we conclude that $p_{i} \nmid \delta_{i}-1$. We already know that $p_{i} \nmid\delta_{i}$, so in both cases we get  that $m_{i}=0$ is the maximal power, completing the proof that $\operatorname{gcd}(z, r)=d$.
\end{proof}

\begin{corollary}
The number of different ideal lattices of \QLS\ among all choices of $\mm$ is
$
\divisors(r)^k
$
for odd $r$ and 
$
\divisors(r/2)^k
$
for even $r$.
\end{corollary}
\begin{corollary}\label{easy34612}
When $r\in\{3,4,6,12\}$, the invariants for equivariant isomorphism of \QLS\ listed in Theorem \ref{Thm:IdealLattice} are complete.
\end{corollary}
\begin{proof}
With such an $r$, with $x,m,n\in\Zru$ we have
\begin{equation}\label{extractm}
\gcd(m-x,r)=\gcd(n-x,r) \Longrightarrow m=n.
\end{equation}
By Lemma \ref{Lemma:simpleiso}, when we have \QLS\ and \QLSalt\ given, we may assume up to equivariant isomorphism that $m_1=n_1=1$, and then show by invoking \eqref{extractm} successively  that when (iii) of Theorem \ref{Thm:IdealLattice} holds, we in fact have $m_i=n_i$ for all $i$.
\end{proof}

\section{Dimension 3}
In order to construct explicit equivariant isomorphisms we will in this section view $C(L_q^3(r;\underline{m}))$ as the fixed point algebra  $C^*(L_3)^{\rho_{\underline{m}}^r}$ of the graph $C^*$-algebra $C^*(L_3)$ that describes the quantum 3-sphere. Hence all statements will be on finding equivariant isomorphisms between $C^*(L_3)^{\rho_{\underline{m}}^r}$ and $C^*(L_3)^{\rho_{\underline{n}}^r}$. 

Let the vertices and edges of the graph $L_3$ and $L_3^{r;(m_1,m_2)}$ be labelled as in Figure \ref{quantum3sphere} and \ref{graphlensspace} respectively. 
Let  $m_i,n_i\in\Zru$. For $t_i,t_j\in \{0,1,\dots,r-1 \}$ let $\alpha_{t_i,t_j}:=e_{ii}^{t_i}e_{ij}e_{jj}^{t_j}$. Then 
$$
{\mathcal{A}}(r;(m_1,m_2))=\{\alpha_{t_2,t_2} : \ m_1(t_1+1)+m_2t_2\equiv 0 \pmod{r}\}
$$
and
$$
\newSStilde{(m_1,m_2)}=\{t_1+t_2+1\lvert \ \alpha_{t_1,t_2}\in \mathcal{A}(r;(m_1,m_2)) \}
$$
$$
\newSS{(m_1,m_2)}=\{t_1+t_2+1 \mod{r} \lvert \ \alpha_{t_1,t_2}\in \mathcal{A}(r;(m_1,m_2)) \}
$$
by Definition \ref{def:multisets}. 

\begin{proposition}\label{PropGenerators3} 
Let $r$ be a positive integer and $m_1,m_2$ be positive integers such that $\gcd(m_i,r)=1$ for $i=1,2$ and let $C^*(\{S_{e_{ii}^r},\ S_{\alpha_{t_1,t_2}} \})$ be the $C^*$-subalgebra of $C^*(L_3)$ generated by the elements 
\begin{equation}\label{generators}
S_{e_{11}^r}, \ S_{e_{22}^r}, \ S_{\alpha_{t_1,t_2}} \ \text{with} \ \alpha_{t_1,t_2}\in  \mathcal{A}(r;(m_1,m_2)).
\end{equation} 
Then there exists an isomorphism 
$$
\phi: C^*(L_3^{r;(m_1,m_2)})\to C^*(\{S_{e_{ii}^r},\ S_{\alpha_{t_1,t_2}} \})\subseteq  C^*(L_3)^{\rho_{\underline{m}}^r}
$$
$$
\begin{aligned}
P_{w_1}&\mapsto S_{e_{11}^r}^*S_{e_{11}^r}=P_{v_1} \\
P_{w_2}&\mapsto S_{e_{22}^r}^*S_{e_{22}^r}=P_{v_2} \\
S_{f_1}&\mapsto S_{e_{11}^r} \\
S_{f_2}&\mapsto S_{e_{22}^r}
\end{aligned}
$$
and $S_{g_i}$ is mapped to one of the $S_{\alpha_{t_1,t_2}}, \alpha_{t_1,t_2}\in  \mathcal{A}(r;(m_1,m_2))$.  

It then follows that $C^*(L_3)^{\rho_{\underline{m}}^r}$ is generated as a $C^*$-algebra by the elements in \eqref{generators}. 
\begin{proof}
We will first show that the image of $\phi$ satisfies the Cuntz-Krieger relations for $L_3^{r;(m_1,m_2)}$. Then the map $\phi$ will be a $*$-homomorphism which is clearly surjective. 
The Cuntz-Krieger relations for $L_3^{r;(m_1,m_2)}$ are as follows: 
\begin{equation}
\begin{aligned}\label{relation1}
S_{f_1}^*S_{f_1}&=P_{w_1}, \\
S_{f_2}^*S_{f_2}&=S_{f_2}S_{f_2}^*=P_{w_2}, \\
S_{g_i}^*S_{g_i}&=P_{w_2}, i=1,\dots,r \\
S_{f_1}S_{f_1}^*&\leq P_{w_1}, \ S_{g_i}S_{g_i}^*\leq P_{w_1}, i=1,\dots,r
\end{aligned}
\end{equation}
\begin{equation}\label{relation2}
P_{w_1}=S_{f_1}S_{f_1}^*+\sum_{i=1}^r S_{g_i}S_{g_i}^*
\end{equation}
It is clear that the image satisfies the relations in \eqref{relation1} since for any finite path $\alpha$ we have $S_{\alpha}S_{\alpha}^*\leq P_{s(\alpha)}$ and $S_{\alpha}^*S_{\alpha}=P_{r(\alpha)}$. The image of $\phi$ satisfies \eqref{relation2} by the following calculation: 
$$
\begin{aligned}
&S_{e_{11}^r}S_{e_{11}^r}^*+S_{e_{11}^{r-1}}S_{e_{12}}S_{e_{12}}^*S_{e_{11}^{r-1}}^*+\sum_{\substack{t=0 \\ t_2=-m_2^{-1}m_1(t_1+1) \mod{r}}}^{r-2}S_{e_{11}^{t_1}}S_{e_{12}}S_{e_{22}^s}S_{e_{22}^{t_2}}^*S_{e_{12}}^*S_{e_{11}^{t_1}}^* \\
&=S_{e_{12}}S_{e_{12}}^*+S_{e_{11}^r}S_{e_{11}^r}^*+S_{e_{11}^{r-1}}S_{e_{12}}S_{e_{12}}^*S_{e_{11}^{r-1}}^*+\sum_{t_1=1}^{r-2}S_{e_{11}^t}S_{e_{12}}S_{e_{12}}^*S_{e_{11}^t}^* \\
&=S_{e_{12}}S_{e_{12}}^*+S_{e_{11}^{r-1}}(S_{e_{11}}S_{e_{11}}^*+S_{e_{12}}S_{e_{12}}^*)S_{e_{11}^{r-1}}^*+\sum_{t_1=1}^{r-2}S_{e_{11}^{t_1}}S_{e_{12}}S_{e_{12}}^*S_{e_{11}^{t_1}}^* \\ 
&=S_{e_{12}}S_{e_{12}}^*+S_{e_{11}^{r-1}}P_{v_1}S_{e_{11}^{r-1}}^*+S_{e_{11}^{r-2}}S_{e_{12}}S_{e_{12}}^*S_{e_{11}^{r-2}}^*+\sum_{t_1=1}^{r-3}S_{e_{11}^{t_1}}S_{e_{12}}S_{e_{12}}^*S_{e_{11}^{t_1}}^* \\ 
&=S_{e_{12}}S_{e_{12}}^*+S_{e_{11}^{r-2}}(S_{e_{11}^r}S_{e_{11}^r}^*+S_{e_{12}}S_{e_{12}}^*)S_{e_{11}^{r-2}}^*+\sum_{t_1=1}^{r-3}S_{e_{11}^{t_1}}S_{e_{12}}S_{e_{12}}^*S_{e_{11}^{t_1}}^* \\ 
&=\cdots = S_{e_{12}}S_{e_{12}}^*+S_{e_{11}^2}S_{e_{11}^2}^* +S_{e_{11}}S_{e_{12}}S_{e_{12}}^*S_{e_{11}}^*\\
&=  S_{e_{12}}S_{e_{12}}^*+S_{e_{11}}(S_{e_{12}}S_{e_{12}}^*+S_{e_{11}}S_{e_{11}}^*)S_{e_{11}}^* =P_{v_1}
\end{aligned}
$$
For injectivity we cannot apply the Cuntz-Krieger uniqueness theorem since $f_{2}$ is a vertex-simple loop without an exit. Instead we will apply the generalised Cuntz-Krieger uniqueness theorem presented in \cite{ws:gckut}. We clearly have that $\phi(P_{w_i})\neq 0$ for $i=1,2$ hence it follows from \cite[Theorem 1.2]{ws:gckut} that $\phi$ is injective if and only if the spectrum of $\phi(S_{f_{2}})=S_{e_{22}^r}$ contains the entire unit circle. By the last part of the proof of Theorem 2.4 in \cite{akdpir:ckadg} it follows that this is indeed the case.

Since $C^*(L_3)^{\rho_{\underline{m}}^r}\cong C^*(L_3^{r,\underline{m}})$ by \cite{jhhws:qlsga} we obtain that $C^*(L_3)^{\rho_{\underline{m}}^r}$ is indeed isomorphic to the $C^*$-algebra generated by the elements in \eqref{generators}.  

\end{proof}
\end{proposition}

From Proposition \ref{PropGenerators3} we immediately see that for two sets of weights $\underline{m}$ and $\underline{n}$ we can construct an isomorphism from $C^*(L_3)^{\rho^r_{\underline{m}}}$ to $C^*(L_3)^{\rho^r_{\underline{n}}}$ as follows. First, to distinguish the two fixed point algebras we denote the vertices and edges in $L_3$ as in Figure \ref{quantum3sphere} when we consider the fixed point algebra under $\rho^r_{\underline{m}}$. For the fixed point algebra under $\rho^r_{\underline{n}}$ we denote the edges by $h_{ij}$, moreover we let $\beta_{k_1,k_2}:=h_{11}^{k_1}h_{12}h_{22}^{k_2}$ for $k_1,k_2\in\{0,\dots,r\}$ such that $n_1(k_1+1)+n_2k_2\equiv 0 \pmod{r}$. Then the following is an isomorphism from $C^*(L_3)^{\rho^r_{\underline{m}}}$ to $C^*(L_3)^{\rho^r_{\underline{n}}}$:
\begin{equation}\label{AnIsomorphism}
\begin{aligned}
S_{e_{11}^r}\mapsto S_{h_{11}^r},  \  S_{e_{22}^r}\mapsto S_{h_{22}^r}, \
S_{\alpha_{t_1,t_2}}\mapsto S_{\beta_{k_1,k_2}}
\end{aligned}
\end{equation}
Note that we have a freedom in the isomorphism to pair $S_{\alpha_{t_1,t_2}}$ with $S_{\beta_{k_1,k_2}}$ for any choice of $k_1,k_2$, but these isomorphisms would in many cases not preserve the gauge action inherited from $C^*(L_3)$. 

Since $\gamma_w(S_{\alpha_{t_1,t_2}})=w^{t_1+t_2+1}S_{\alpha_{t_1,t_2}}$ we see immediately that if $\newSStilde{(m_1,m_2)}$ equals $\newSStilde{(n_1,n_2)}$, then we can pair $S_{\alpha_{t_1,t_2}}$ with a $S_{\beta_{k_1,k_2}}$ such that $t_1+t_2+1=k_1+k_2+1$ and the gauge action is then preserved. We will in Proposition \ref{ExistenceGPisom} see that this condition is too restrictive, instead have to use the multiset $\newSS{\underline{m}}$.
\\

We will now describe $\newSS{\underline{m}}$ further, which will become crucial in the construction and existence of an equivariant isomorphism.  
\begin{proposition}\label{PropSetOfNumbers}
Let $r$ be a positive integer and $m_1,m_2$ be positive integers such that $\gcd(m_i,r)=1$ for $i=1,2$. Then we have 
\begin{equation}\label{SetOfNumbers}
 \begin{aligned}
 \newSS{(m_1,m_2)}=\left\{\gcd(m_2-m_1,r)k: k=1,2,\dots,\frac{r}{\gcd(m_2-m_1,r)} \right\}.
 \end{aligned}
 \end{equation}
\begin{proof}
For $t_1=0,\dots,r-1$ we have 
\begin{eqnarray*}
t_1+t_2+1&=&t_1-m_2^{-1}m_1(t_1+1)+1 \pmod{r} \\
&=&(t_1+1)(1-m_2^{-1}m_1) \pmod{r} \\
&=&c\gcd(1-m_2^{-1}m_1,r)\alpha \pmod{r}
\end{eqnarray*}
where $1-m_2^{-1}m_1=\gcd(1-m_2^{-1}m_1,r)\alpha$ for a unit $\alpha$ in $\Zr$ and $c=1,\dots,r$. For a number $t$ between $0$ and $r$ we have that $r=\gcd(1-m_2^{-1}m_1,r)t$. Hence $t\gcd(1-m_2^{-1}m_1,r)\alpha\equiv 0 \pmod{r}$ and we only have to consider the numbers for which $c\leq t$. 

Moreover, since $\alpha$ is a unit, the set of numbers
$$
\{c\gcd(1-m_2^{-1}m_1)
,r)\alpha \mod{r}: c=1,\dots,r\}
$$
is the same as the numbers $k\gcd(1-m_2^{-1}m_1),r)$, $ k=1,\dots,\frac{r}{\gcd(1-m_2^{-1}m_1,r)}$. Since $\gcd(1-m_2^{-1}m_1,r)=\gcd(m_1-m_2,r)$ we obtain \eqref{SetOfNumbers}. 
\end{proof}
\end{proposition}

\begin{remark}
Proposition \ref{PropSetOfNumbers} shows moreover that if $|\newSS{(m_1,m_2)}|=|\newSS{(n_1,n_2)}|$ for two weight vectors $\underline{m}$ and $\underline{n}$ then $\newSS{(m_1,m_2)}=\newSS{(n_1,n_2)}$. Hence the multisets can be distinguished by only considering the number of elements in the sets and not the particular elements. Moreover, note that $|\newSS{(m_1,m_2)}|$ equals the number of periodic lines in level $2$ in the graphs (see e.g. Figure \ref{F8}) which by Proposition \ref{PropSetOfNumbers} is precisely $\frac{r}{\gcd(m_1-m_2,r)}$.
\end{remark}

\begin{example}
We obtain by Proposition \ref{PropSetOfNumbers} that $\newSS[5]{1,3}=\newSS[5]{2,3}=\{1,2,3,4\}$ and $\newSS[5]{2,2}=\{0\}$ which again indicates the similarity between $F_{5;(1,3)}$ and $F_{5;(2,3)}$ and differences with the graph $F_{5;(2,2)}$, as seen in Example \ref{exgraph2}. 
\end{example}

\begin{example}\label{ExGaugeIsomorphism}
Consider $r=5$ and the set of weights $\underline{m}=(1,3)$ and $\underline{n}=(2,3)$. From Figure \ref{F13} and \ref{F23} we can determine the value of $t_2$ corresponding to a given $t_1\in\{0,\dots,r-1\}$. The numbers are as follows: 
\begin{table}[H]
    \centering
    \begin{tabular}{c c|c c c c c}
    & $t_1$  & 0 & 1 & 2 & 3 & 4\\
    \hline \hline 
    $\underline{m}=(1,3)$ & $t_2$ & 3 & 1 & 4 & 2 & 0   \\  \hline 
     $\underline{n}=(2,3)$ & $t_2$ & 1 & 2 & 3 & 4 & 0 
    \end{tabular}
\end{table}
Consider the following map from $C^*(L_3)^{\rho^5_{(1,3)}}$ to $C^*(L_3)^{\rho^5_{(2,3)}}$: 
$$
\begin{aligned}
S_{e_{11}^5}&\mapsto S_{h_{11}^5}, \ S_{e_{22}^5}\mapsto S_{h_{22}^5}, \\
S_{\alpha_{0,3}}&\mapsto S_{\beta_{1,2}}, \ S_{\alpha_{1,1}}\mapsto S_{\beta_{3,4}}S_{h_{22}^5}^* \\
S_{\alpha_{2,4}}&\mapsto S_{\beta_{0,1}}S_{h_{22}^5}, \ S_{\alpha_{3,2}}\mapsto S_{\beta_{2,3}}, \\
S_{\alpha_{4,0}}&\mapsto S_{\beta_{4,0}}. 
\end{aligned}
$$
The map is an isomorphism since the right hand side satisfies the Cuntz-Krieger relations for $L_3^{r;\underline{m}}$ and it is clearly equivariant. The example illustrates why we wish to calculate the numbers $t_1+t_2+1$ modulo $r$ since we have the freedom to multiply by the elements $S_{h_{22}^5}$ and $S_{h_{22}^5}^*$ which correspond to changing the generators of $C^*(L_3)^{\rho^5_{(2,3)}}$. 
\end{example}

\begin{proposition}\label{ExistenceGPisom}
Let $\underline{m},\underline{n}\in\Nb^2$ be such that $\newSS{(m_1,m_2)}=\newSS{(n_1,n_2)}$. Let $S_\nu$ be a generator in $C^*(L_3)^{\rho_{\underline{m}}^r}$ and $S_\mu$ a generator of $C^*(L_3)^{\rho_{\underline{n}}^r}$ for which $s(\nu)=s(\mu), r(\nu)=r(\mu)$ and $|\nu|\equiv |\mu| \pmod{r}$. If $|\nu|=|\mu|+\ell r$ for $\ell \in\Nb$ let 
$$
S_{\nu}\mapsto S_{\mu}S_{e_{22}^{\ell r}}
$$
and if $|\nu|=|\mu|-\ell r$ for $\ell\in\Nb$ let 
$$S_{\nu}\mapsto S_{\mu}S_{e_{22}^{\ell r}}^{*}.$$
Then the above defines an equivariant isomorphism between $C^*(L_3)^{\rho_{\underline{m}}^r}$ and $C^*(L_3)^{\rho_{\underline{n}}^r}$. 
\begin{proof}
     The map is clearly a $*$-homomorphism since the domain and range satisfy the same defining relations which follows since $v_2$ is a sink. 
\end{proof}
\end{proposition}

\begin{theorem}\label{mainresult3}
Let $r\in\Nb$ and consider $m_i,n_i\in\Zru$ for $i=1,2$. Then the following are equivalent: 
\begin{enumerate}[(i)]
\item $(\QLS[3],\gamma)\simeq (\QLSalt[3],\gamma)$
\item $\gcd(m_2-m_1,r)=\gcd(n_2-n_1,r)$
\item $\newSS{\mm}=\newSS{\nn}$ (See notation \ref{WreducedTheSame})
\end{enumerate}
\end{theorem}
\begin{proof} 
We have already noted, after Theorem \ref{Thm:IdealLattice}, that $(i)$ implies $(ii)$. That $(ii)$ is equivalent to $(iii)$ follows directly by Proposition \ref{PropSetOfNumbers}, and that   ($iii)$ implies $(i)$ follows by Proposition \ref{ExistenceGPisom}. 
\end{proof}

\begin{corollary}
When $k=1$, the invariants for equivariant isomorphisms of \QLS\ listed in Theorem \ref{Thm:IdealLattice} are complete.
\end{corollary}

\section{Dimension quadruples}

This far, we have seen that the ideal lattice of the fixed point algebras provide us with a useful and easily computable invariant, which is complete when $k=1$ or when $r\in\{3,4,6,12\}$. However, we shall see that it is not complete in general, and even though -- as we have already seen and used -- the multisets $\newSS{\mm}$ are useful tools for further analysis, we do not know how to produce complete invariants from them in general.

Hence we now introduce the \emph{dimension quadruples} which are, conjecturally (cf. \cite{rh:ggclpa}, \cite{seer:rmsigc}), complete invariants for equivariant isomorphism of graph $C^*$-algebras equipped with their canonical gauge action. We follow \cite{seer:rmsigc} but need to adjust the definitions to the fact that the quantum lens spaces are equivariantly given not as graph $C^*$-algebras but as corners of them.
\\

The starting point for this approach is to note that since $E\times_1\ZZ$ has no cycles, $C^*(E\times_1\ZZ)$ is an AF algebra, and consequently  $K_0(C^*(E\times_1\ZZ))$ with its canonical order structure is a dimension group. We also see that sending each vertex and edge of $E\times_1\ZZ$ one step to the right defines an automorphism $\lt:C^*(E\times_1\ZZ)\to C^*(E\times_1\ZZ)$ - a right translation - which induces an automorphism of $K_0(C^*(E\times_1\ZZ))$. For any positive element $x$ of an ordered group $G$, $I(x)$ denotes its \emph{order ideal}, the set
\[
I(x)=\{y\in G :\exists n: 0\leq y\leq nx\}.
\]
\begin{definition}
When $E$ is a graph with finitely many vertices, the \emph{dimension triple} is defined as
\[
\DT(E)=(K_0(C^*(E\times_1\ZZ)),K_0(C^*(E\times_1\ZZ))_+,\lt_*)
\]
Let $p=\sum_{v\in V}p_v\in C^*(E)$ be an orthogonal  sum of vertex projections. We define
\[
p_0^E:= \sum_{v\in V}p_{(0,v)}\in C^*(E\times _1\ZZ)
\]
and two kinds of \emph{dimension quadruples} of $(E,p)$ as
\[
\DQp(E,p)=(\DT(E),[p_0^E])
\]
and
\[
\DQ(E,p)=(\DT(E),I([p_0^E])),
\]
respectively
\end{definition}

\begin{theorem}\label{Thm:DQ0}
When $(pC^*(E)p,\gamma)\simeq (qC^*(F)q,\gamma)$, with $p$ and $q$ {sums of vertex projections that are full}, we have
\[
\DQp(E,p)\simeq \DQp(F,q).
\]
When $(pC^*(E)p \otimes\KKK,\gamma\otimes\id)\simeq (qC^*(F)q \otimes\KKK,\gamma\otimes\id)$, we have 
\[
\DQ(E,p)\simeq \DQ(F,q).
\]
\end{theorem}
\begin{proof}
We can apply \cite[Lemma 3.2]{seeras:agciir} to see that the dimension triples of the full corners agree with the dimension triples of the whole graph $C^*$-algebras. Arguing from here as in 
\cite{seer:rmsigc} we get that the dimension quadruples agree.
\end{proof}

We abbreviate $\DT(r;\mm):=\DT(\ourE)$,  $\DQp(r;\mm):=\DQp(\ourE,p)$, $\DQ(r;\mm):=\DQ(\ourE,p)$.

\newcommand{\ee}{\underline{e}}
\newcommand{\xx}{\underline{x}}

The maps/matrices $\sh,\perm_\alpha:\ZZ^r\to \ZZ^r$ ($\alpha\in\Zru$) are given by
\[
\sh\ee_i=\ee_{i-1}\qquad \perm_\alpha\ee_i=\ee_{\alpha i}
\]
Note that $\sh\perm_\alpha=\perm_\alpha\sh^\alpha$.

\begin{lemma}\mbox{}\label{describeDQ}
\begin{enumerate}[(i)]
\item The map $[p_{((i,j),0)}]\mapsto \ee_{(i,j)}$ induces an isomorphism $K_0(C^*(\ourE \times_1\ZZ))\simeq \Z^{r(k+1)}$.
\end{enumerate}
Under the isomorphism in (i),
\begin{enumerate}[(i)]\addtocounter{enumi}{1}
\item  $K_0(C^*(\ourE \times_1\ZZ))_+$ is taken to 
\[
\left\{\xx\mid  \forall i,j: x_{(i,j)}\geq 0\vee(\exists i',j': x_{(i',j')}>0 \wedge (i',j')\myto (i,j)) \right\}
\]
\item $\lt_*$ is taken to 
\[
\myA[r;\mm]:=\begin{bmatrix}
\sh^{m_1}&&&&\\
-\sh^{m_2}&\sh^{m_2}&&&\\
-\sh^{m_3}&-\sh^{m_3}&\sh^{m_3}&&\\
\vdots&&&\ddots\\
-\sh^{m_{k+1}}&-\sh^{m_{k+1}}&-\sh^{m_{k+1}}&\dots&\sh^{m_{k+1}}
\end{bmatrix}
\]
\item $[P_0]$ is taken to $\sum_{i=1}^{k+1}\ee_{(i,0)}$.
\item $I([P_0])$ is taken to the image of $K_0(C^*(\ourE \times_1\ZZ))_+$ intersected with
\[
\left\{\xx\mid  \forall 0<j<r: x_{(0,j)}=0 \right\}
\]
\end{enumerate}
\end{lemma}
\begin{proof}
It is standard to compute the ordered $K_0$-group, cf. \cite{mt:okgc} and/or \cite{rh:ggclpa}. We get from the structure of $\ourE\times_1\ZZ$ combined with the Cuntz-Krieger relations that
\[
[p_{((v_i,j),\ell)}]=\sum_{i'\geq i}[p_{(v_{i'},j+m_i),\ell+1)}].
\]
We have in particular that $[p_{((v_{k+1},j),\ell)}]=[p_{((v_{k+1},j+m_{k+1}),\ell+1)}]$, and starting from here and solving one level at a time  allows us to specify how the $[p_{((v_i,j),0)}]$ generate all the $K$-classes $[p_{((v_i,j),\ell)}]$ for $\ell\geq 1$. In particular, since $\operatorname{rt}_*([p_{(v_i,j),0)}])=[p_{(v_i,j),1)}]$, we can compute $\myA[r;\mm]$ this way. The remaining claims are clear.
\end{proof}

%
%
%
%

\begin{proposition}\label{charH}
Suppose an isomorphism $\eta:\DT(r;\mm)\to \DT(r;\nn)$ is induced by $H\in \matrM_{r(k+1)}(\ZZ)$. With
\[
\PP:=\PP(r;\mm)=\PP(r;\nn)\qquad \PP_0:=\PP_0(r;\mm)=\PP_0(r;\nn)
\]
we get that 
\[
H=\begin{bmatrix}
\sh^{\ell_1}\perm_{m_1/n_1}&&&&\\
Y_{21}&\sh^{\ell_2}\perm_{m_2/n_2}&&&\\
Y_{31}&Y_{32}&\sh^{\ell_3}\perm_{m_3/n_3}&&&\\
\vdots&\vdots&&\ddots\\
Y_{{k+1},1}&Y_{{k+1},2}&\dots&& \sh^{\ell_{k+1}}\perm_{m_{k+1}/n_{k+1}}
\end{bmatrix}
\]
for all $1\leq c\leq k+1$, where
\[
Y_{ab}(i,j) =0
\]
whenever $1\leq b<a\leq k+1$ and $(b,j)\not\myto (a,i)$ in $\PP$.

We have that such an $H$ induces an isomorphism $\DQ(r;\mm)\to \DQ(r;\nn)$ if and only if $\ell_1=0$, and that it induces an isomorphism $\DQp(r;\mm)\to \DQp(r;\nn)$
whenever
\begin{equation}\label{H1eq}
H \sum_{i=1}^{k+1}\ee_{(i,0)}= \sum_{i=1}^{k+1}\ee_{(i,0)}
\end{equation}
\end{proposition}
\begin{proof}
We let $H$ be given with the necessary property
\[
H\myA[r;\mm]=\myA[r;\nn]H
\]
It follows directly from Lemma \ref{describeDQ}(ii), which describes the order structure induced on $\ZZ^{r(k+1)}$, that any such matrix must be lower triangular, and that every diagonal block matrix $Y_{cc}$ must be a permutation matrix. With $\ell_c$ the index of the nonvanishing entry of the first column of $Y_{cc}$ we get from the requirement that $Y_{cc}\sh^{m_c}=\sh^{n_c}Y_{cc}$ that $Y_{cc}$ has the stipulated form.
It also follows from the order structure that 
\[
Y_{ab}(i,j)\geq 0
\]
in the cases considered for $H$ to send the canonical basis vectors, which are all positive, to positive elements. The same argument applied to $H^{-1}$ shows the other inequality.

The two final claims follow from  (iv) and (v) of Lemma \ref{describeDQ}.
\end{proof}

We now, for later use, analyze the $k=1$ more carefully. We easily see, by the results below, that there is no new information in the dimension quadruples in this case, but will use a detailed understanding of the limited freedom of choice in the $Y_{21}$ matrix in an $H$ implementing the isomorphism of dimension quadruples to study cases of larger $k$, mainly under the added condition that $r$ is a prime.

\begin{theorem}\label{Thm:IdealLatticeplus} The statement  
\begin{enumerate}[(i)]\addtocounter{enumi}{5}
\item $C^*(\ourE\times_1\ZZ)\cong C^*(\ourEalt\times_1\ZZ)$
\end{enumerate}
is equivalent to (i)--(v) of Theorem \ref{Thm:IdealLattice} for fixed $r$ and $k$ and varying $m_i,n_i\in\Zru$:
\begin{proof}
When (iv) holds, the orderings on $K_0(C^*(\ourE\times_1\ZZ))$ and $K_0(C^*(\ourEalt\times_1\ZZ))$ agree by Proposition \ref{charH}, and since AF algebras are classified up to stable isomorphism by their ordered $K_0$-group, we get that $C^*(\ourE\times_1\ZZ)\otimes \KKK\simeq C^*(\ourEalt\times_1\ZZ)\otimes \KKK$. But these $C^*$-algebras are stable, so (vi) follows. Obviously, (vi) implies (ii).
\end{proof}
\end{theorem}

\begin{theorem} \label{mainresult3plus} The conditions in Theorem \ref{mainresult3} are equivalent to
\begin{enumerate}[(i)]\addtocounter{enumi}{3}
\item $\DQ(r;\mm)\simeq \DQ(r;\nn)$
\end{enumerate}
\begin{proof}
By Theorem \ref{Thm:DQ0} we get that (i) implies (iv). In the other direction, when (iv) holds, then Theorem \ref{Thm:IdealLatticeplus} applies to show (ii).
\end{proof}
\end{theorem}

\begin{definition}
$\Delta(Z)_i=\operatorname{trace}(\sh^iZ)$.
\end{definition}

%

The following result is  elementary.

\begin{lemma}\label{SYYSZ}
\mbox{}
\begin{enumerate}[(i)]
\item When 
\[
\sh^iY-Y\sh^j =Z
\]
for $i,j\in \Zru$ then $Y$ is completely determined by its first column.
\item There is a matrix 	$Y$  solving
\[
\sh Y-Y\sh =Z
\]
precisely when $\Delta(Z)=\oo$. 
\item When $Y$ is the unique solution with first column $(y_{i})$ according to (i) and (ii), 
\[
Y(i,j)=\sum_{k=1}^j z_{i-j-1+k,k}+y_{i-j}
\]
\end{enumerate}
\end{lemma}

%
%
%
%


\begin{lemma}\label{describeH}
When $r$ is a prime, and $\DQ(r;(m_1,m_2))\simeq \DQ(r;(n_1,n_2))$, then there is an isomorphism given on the form
\[
H=\begin{bmatrix}\perm_{m_1/n_1}&0\\Y_{21}&\perm_{m_2/n_2}\end{bmatrix}
\]
with $Y_{21}$ determined uniquely by its first column $(y_i)$.
When $m_1\not=m_2$, we  have
\[
Y_{21}(i,j)=\iverson{\frac{i/n_2-j/m_1}{n_1/n_2-1}\leq \ell/m_1}-\iverson{\frac{i/n_2-j/m_1}{m_2/m_1-1}\leq \ell/m_1}+y_{i/n_2-j/m_1}
\]
and when $m_1=m_2$, we have
\[
Y_{21}(i,j)=y_i
\]
\end{lemma}
\begin{proof}
Assume first that $m_1\not=m_2$. We start by showing that for any choice of $\ell_1,\ell_2$, there is a choice of a matrix $X_{21}$ so that 
\[
\begin{bmatrix}\sh^{\ell_1}&0\\X_{21}&\sh^{\ell_2}\end{bmatrix}
\]
induces an automorphism on $\DT(r;(m_1,m_2))$. We note that it suffices to show the existence of $\widetilde{X}_{21}$  so that 
\[
\begin{bmatrix}\sh^{\ell_1}&0\\\widetilde{X}_{21}&\sh^{\ell_2}\end{bmatrix}:\DT(r;(1,m_2/m_1))\to \DT(r;(m_1/m_2,1))
\]
induces an isomorphism, since in that case we can let
\[
X_{21}(i,j)=\widetilde{X}_{21}(i/m_2,j/m_1)
\]
and note that
\[
\xymatrix{
\DT(r;(m_1,m_2))\ar[rr]^-{\left[\begin{smallmatrix}\sh^{\ell_1}&0\\X_{21}&\sh^{\ell_2}\end{smallmatrix}\right]}\ar[dd]_-{\left[\begin{smallmatrix}\perm_{1/m_1}&0\\0&\perm_{1/m_1}
\end{smallmatrix}\right]}&&\DT(r;(m_1,m_2)\\
&&\\
\DT(r;(1,m_2/m_1))\ar[rr]_-{\begin{array}{c}\mbox{}\\\left[\begin{smallmatrix}\sh^{\ell_1m_2}\perm_{m_1/m_2}&0\\\widetilde{X}_{21}&\sh^{\ell_2m_2}\perm_{m_1/m_2}\end{smallmatrix}\right]\end{array}}&&\DT(r;(m_1/m_2,1))\ar[uu]_-{\left[\begin{smallmatrix}\perm_{m_2}&0\\0&\perm_{m_2}\end{smallmatrix}\right]}
}
\]
commutes. The map induced will be an order isomorphism for any choice of $\widetilde{X}_{21}$, but we need to establish that 
\[
\left[\begin{matrix}\sh^{\ell_1m_2}\perm_{m_1/m_2}&0\\\widetilde{X}_{21}&\sh^{\ell_2m_2}\perm_{m_1/m_2}\end{matrix}\right]
\begin{bmatrix}
\sh&0\\
-\sh^{m_1/m_2}&\sh^{m_1/m_2}\end{bmatrix}
=
\begin{bmatrix}
\sh^{m_2/m_1}&0\\
-\sh&\sh\end{bmatrix}
\left[\begin{matrix}\sh^{\ell_1m_2}\perm_{m_1/m_2}&0\\\widetilde{X}_{21}&\sh^{\ell_2m_2}\perm_{m_1/m_2}\end{matrix}\right]
\]
which comes out to the condition 
\[
\widetilde{X}_{21}\sh-\sh^{\ell_2m_2+1}\perm_{m_1/m_2}=-\sh^{\ell_1m_2+1}\perm_{m_1/m_2}+\sh\widetilde{X}_{21}
\]
at the 21 entry. This is equivalent to 
\[
\sh\widetilde{X}_{21}-\widetilde{X}_{21}\sh=-\sh^{\ell_2m_2+1}\perm_{m_1/m_2}+\sh^{\ell_1m_2+1}\perm_{m_1/m_2}
\]
which has a solution by Lemma \ref{SYYSZ}(ii) because $\Delta(\perm_{m_1/m_2})$ is constant with the value one in each entry.

Composing with the appropriate automorphism of $\DT(r;(m_1,m_2))$ just proved to exist, we have proved that there is an $H$ as stated. We now set out to describe this $H$ 
under the additional condition that $m_1=n_2=1$.  The intertwining condition becomes
\[
\begin{bmatrix}
\perm_{n_1}&0\\{Y}_{21}&\perm_{1/m_2}
\end{bmatrix}
\begin{bmatrix}
\sh&0\\-\sh^{m_2}&\sh^{m_2}
\end{bmatrix}=\begin{bmatrix}
\sh^{n_1}&0\\-\sh&\sh
\end{bmatrix}\begin{bmatrix}
\perm_{n_1}&0\\{Y}_{21}&\perm_{1/m_2}
\end{bmatrix},
\]
which gives us
\[
{Y}_{21}\sh-\perm_{1/m_2}\sh^{m_2}=-\sh P_{n_1}+\sh \widetilde{}Y_{21}
\]
and consequently we have the situation described in  Lemma \ref{SYYSZ}(ii) with
\[
Z=\sh P_{n_1}-P_{1/m_2}\sh^{m_2}.
\]
This matrix has the  entries
\[
Z(i,j)=\iverson{i=n_1j-1}-\iverson{i=j/m_2-1}
\]
so appealing to  Lemma \ref{SYYSZ}(iii) we get
\begin{eqnarray*}
Y_{21}(i,j)&=&\sum_{k=1}^j\left(\iverson{i-j-1+k=n_1k-1}-\iverson{i-j-1+k=k/m_2-1}\right)+y_{i-j}\\
&=&\sum_{k=1}^j\left(\iverson{k=\frac{i-j}{n_1-1}}-\iverson{k=\frac{i-j}{1/m_2-1}}\right)+y_{i-j}\\
&=&\iverson{1\leq \frac{i-j}{n_1-1}\leq j}-\iverson{1\leq \frac{i-j}{1/m_2-1}\leq j}+y_{i-j}\\
&=&\iverson{ \frac{i-j}{n_1-1}\leq \ell}-\iverson{ \frac{i-j}{1/m_2-1}\leq j}+y_{i-j}\\
\end{eqnarray*}

As above, we reduce from the general to the special case by the fact that 
\[
\xymatrix{
\DQ(r;(m_1,m_2))\ar[rr]^-{\left[\begin{smallmatrix}\perm_{n_1/m_1}&0\\Y_{21}&\perm_{m_2}\end{smallmatrix}\right]}\ar[dd]_-{\left[\begin{smallmatrix}\perm_{1/m_1}&0\\0&\perm_{1/m_1}
\end{smallmatrix}\right]}&&\DQ(r;(n_1,n_2))\\
&&\\
\DQ(r;(1,m_2/m_1))\ar[rr]_-{\begin{array}{c}\mbox{}\\\left[\begin{smallmatrix}\perm_{n_1/m_1}&0\\\widetilde{Y}_{21}&\perm_{n_2/m_2}\end{smallmatrix}\right]\end{array}}&&\DQ(r;(n_1/n_2,1))\ar[uu]_-{\left[\begin{smallmatrix}\perm_{n_2}&0\\0&\perm_{n_2}
\end{smallmatrix}\right]}\\
}
\]
commutes with
\[
Y_{21}(i,j)=\widetilde{Y}_{21}(i/n_2,j/m_1).
\]

The claims for $m_1=m_2$ are straightforward. 
\end{proof}

\section{Dimension 5, $r$ prime, sufficiency}



%
%
%
In order to construct explicit equivariant isomorphisms we will in this section view $C(L_q^5(r;\underline{m}))$ as the fixed point algebra  $C^*(L_5)^{\rho_{\underline{m}}^r}$ of the graph $C^*$-algebra $C^*(L_5)$ that describes the quantum 5-sphere. Hence all statements will be on finding equivariant isomorphism between $C^*(L_5)^{\rho_{\underline{m}}^r}$ and $C^*(L_5)^{\rho_{\underline{n}}^r}$. 

Let the vertices and edges in the graph $L_5$ be denoted as in Figure \ref{quantum5sphere}. Let $t_1,t_2,t_3\in \{0,1,\dots,r-1\}$ be such that  $m_1(t_1+1)+m_2(t_2+1)+m_3t_3\equiv 0 \pmod{r}$ and consider $\alpha_{t_1,t_2,t_3}=e_{11}^{t_1}e_{12}e_{22}^{t_2}e_{23}e_{33}^{t_3}$.    Under this condition we have
$$
t_2\equiv -m_2^{-1}m_1(t_1+1)-m_2^{-1}m_3t_3-1 \pmod{r}. 
$$
Let $t_2':=-m_2^{-1}m_1(t_1+1) \mod{r}$ and $t_2'':=-m_2^{-1}m_3t_3-1 \mod{r}$ then 
\begin{equation}\label{tandk}
m_1(t_1+1)+m_2t_2'\equiv 0 \Mod{r} \hspace{0.2cm} \text{and} \hspace{0.2cm}  m_2(t_2''+1)+m_3t_3\equiv 0 \Mod{r}
\end{equation}
and $t_2\equiv t_2'+t_2''\Mod{r}$. 
Hence, all triples of numbers $(t_1,t_3,t_3)$ such that $m_1(t_1+1)+m_2(t_2+1)+m_3t_3\equiv 0 \pmod{r}$ comes from two pairs $(t_1,t_2')$ and $(t_2'',t_3)$ satisfying the assumptions in \eqref{tandk}. 

If $t_2'+t_2''<r$, then we clearly have $\alpha_{t_1,t_2,t_3}=\alpha_{t,_1t_2'}\alpha_{t_2'',t_3}$ and therefore they will not be included in the set $\mathcal{A}(r;(m_1,m_2,m_3))$ . If $t_2'+t_2''\geq r$, then we cannot write $\alpha_{t_1,t_2,t_3}$ as a combination of these paths since it would imply that $t_2\geq r$ which is against the assumption. Note that it is not possible to use a power of the generator $S_{e_{22}^r}$ in between $S_{\alpha_{t_1,t_2'}}$ and $S_{\alpha_{t_2'',t_3}}$ to reduce $t_2'+t_2''$ to a number less than $r$ since $v_2$ is not a sink. Then 
$$
\mathcal{A}(r;(m_1,m_2,m_3))=\{\alpha_{t_1,t_2,t_3} | \ m_1(t_1+1)+m_2(t_2+1)+m_3t_3\equiv 0 \pmod{r} \ \text{and} \ t_2'+t_2''\geq r\}
$$
and 
$$
\newSS{(m_1,m_2,m_3))}=\{t_1+t_2+t_3+2\lvert \alpha_{t_1,t_2.t_3}\in \mathcal{A}(r;(m_1,m_2,m_3)) \}.
$$
by Definition \ref{def:multisets}. 

In the following proposition the vertices and edges of $L_5^{r;(m_1,m_2,m_3)}$ are as in Figure \ref{graphlensspace5}.
\begin{proposition}\label{PropGenerators5}
Let $r$ be a positive integer and $m_1,m_2,m_3$ be positive integers such that $\gcd(m_i,r)=1$ for $i=1,2,3$ and let $C^*(\{S_{e_{ii}^r}, S_{\alpha})$ be the $C^*$-subalgebra of $C^*(L_5)$ generated by the elements 
\begin{equation}\label{generators1}
\begin{aligned}
&S_{e_{ii}^r}, i=1,2,3, \\
&S_{\alpha}, \alpha\in \mathcal{A}(r;(m_i,m_j))  ,1\leq i<j\leq 3, \\
&S_{\alpha}, \alpha\in \mathcal{A}(r;(m_1,m_2,m_3))
\end{aligned}
\end{equation}
Then there exists an isomorphism 
$$
 C^*(L_5^{r;(m_1,m_2,m_3)})\to C^*(\{S_{e_{ii}^r}, S_{\alpha}\})\subseteq C^*(L_5)^{\rho_{\underline{m}}^r}
$$
$$
\begin{aligned}
P_{w_i}&\mapsto S_{e_{ii}^r}^*S_{e_{ii}^r}=P_{v_i}, \ S_{f_i}\mapsto S_{e_{ii}^r}, \ i=1,2,3, 
\end{aligned}
$$
and $S_{g_{12}^i}, S_{g_{23}^i}$ are mapped to one of the $S_{\alpha_{t_1,t_2}}, S_{\alpha_{t_2,t_3}}$ respectively for $\alpha_{t_i,t_j}\in  \mathcal{A}(r;(m_i,m_j))$. The $S_{g_{13}^i}$ are mapped to one of the $S_{\alpha_{t_1,t_3}}$ or $S_{\alpha_{t_1,t_2,t_3}}$ with $\alpha_{t_1,t_3}\in  \mathcal{A}(r;(m_1,m_3))$ and $\alpha_{t_1,t_2,t_3}\in  \mathcal{A}(r;(m_1,m_2,m_3))$. 

It then follows that $C^*(L_5)^{\rho_{\underline{m}}^r}$ is generated as a $C^*$-algebra by the elements in \eqref{generators1}. 

\begin{proof}
We will follow the same procedure as in Proposition \ref{PropGenerators3} and show that the image of the map satisfies the Cuntz-Krieger relations for $L_5^{r;\underline{m}}$ given as follows:
\begin{equation}
\begin{aligned}\label{relation11}
S_{f_i}^*S_{f_i}&=P_{w_i}, i=1,2\\
S_{f_3}^*S_{f_3}&=S_{f_3}S_{f_3}^*=P_{w_3}, \\
S_{g^k_{ij}}^*S_{g^k_{ij}}&=P_{w_j}, 0\leq 1 \leq i<j\leq 3, \\
S_{f_i}S_{f_i}^*&\leq P_{w_i}, i=1,2 
\\ S_{g^i_{1j}}S_{g^i_{ij}}^*&\leq P_{w_i},  0\leq 1 \leq i<j\leq 3, 
\end{aligned}
\end{equation}
\begin{equation}\label{relation22}
\begin{aligned}
P_{w_2}&=S_{f_2}S_{f_2}^*+\sum_{i=1}^r S_{g^i_{23}}S_{g^i_{23}}^*,
\end{aligned}
\end{equation}
\begin{equation}\label{relation33}
\begin{aligned}
P_{w_1}&=S_{f_1}S_{f_1}^*+\sum_{i=1}^r S_{g^i_{12}}S_{g^i_{12}}^*
+\sum_{i=1}^{\frac{r(r+1)}{2}} S_{g^i_{13}}S_{g^i_{13}}^*
\end{aligned}
\end{equation}
The relations in \eqref{relation11} follows easily and \eqref{relation22} follows by a similar calculation as in Proposition \ref{PropGenerators3}. To show that the image satisfies \eqref{relation33}, we need to prove that $P_{v_1}$ equals the following in $C^*(L_5)$:
\begin{equation}\label{SumOverPaths}
\begin{aligned}
&S_{e_{11}}^r{S_{e_{11}}^r}^*+ \sum_{\alpha_{t_1,t_2}\in \mathcal{A}(r:(m_1,m_2))} S_{\alpha_{t_1,t_2}}{S_{\alpha_{t_1,t_2}}}^* + \sum_{\alpha_{t_1,t_3}\in \mathcal{A}(r:(m_1,m_3))} S_{\alpha_{t_1,t_3}} {S_{\alpha_{t_1,t_3}}}^* \\
&+ \sum_{\alpha_{t_1,t_2,t_3}\in \mathcal{A}(r:(m_1,m_2,m_3))} S_{\alpha_{t_1,t_2,t_3}} {S_{\alpha_{t_1,t_2,t_3}}}^* 
\end{aligned}
\end{equation}
The above becomes:
$$
\begin{aligned}
&S_{e_{11}}^r{S_{e_{11}}^r}^*+S_{e_{11}}^{r-1}S_{e_{12}}S_{e_{12}}^*{S_{e_{11}}^{r-1}}^* 
+\sum_{\overset{t_1=0}{t_2\equiv -m_2^{-1}m_1(t_1+1) \pmod{r}}}^{r-2} 
S_{e_{11}}^{t_1}S_{e_{12}}S_{e_{22}}^{t_2}{S_{e_{22}}^{t_2}}^*S_{e_{12}}^*{S_{e_{11}}^{t_1}}^* 
\\
&+ S_{e_{11}}^{r-1}S_{e_{13}}S_{e_{13}}^*{S_{e_{11}}^{r-1}}^* +\sum_{\overset{t_1=0}{t_3\equiv -m_3^{-1}m_1(t_1+1) \pmod{r}}}^{r-2} 
S_{e_{11}}^{t_1}S_{e_{13}}S_{e_{33}}^{t_3}{S_{e_{33}}^{t_3}}^*S_{e_{13}}^*{S_{e_{11}}^{t_1}}^* \\
&+ \sum_{\overset{l_1,l_3=0, \ l_2'+l_2''\geq r}{l_2\equiv -m_2^{-1}m_1(l_1+1)-{m_2}^{-1}m_3l_3-1 \pmod{r}}}^{r-2} 
S_{e_{11}}^{l_1}S_{e_{12}}S_{e_{22}}^{l_2}S_{e_{23}}S_{e_{33}}^{l_3}{S_{e_{33}}^{l_3}}^*{S_{e_{23}}}^*{S_{e_{22}}^{l_2}}^*S_{e_{12}}^*{S_{e_{11}}^{l_1}}^*  
\end{aligned}
$$
$$
\begin{aligned}
&=S_{e_{11}}^{r-1}{S_{e_{11}}^{r-1}}^*+\sum_{\overset{t_1=0}{}}^{r-2} 
S_{e_{11}}^{t_1}S_{e_{13}}S_{e_{13}}^*{S_{e_{11}}^{t_1}}^*  
+\sum_{\overset{t_1=0}{t_2\equiv -m_2^{-1}m_1(t_1+1) \pmod{r}}}^{r-2} 
S_{e_{11}}^{t_1}S_{e_{12}}S_{e_{22}}^{t_2}{S_{e_{22}}^{t_2}}^*S_{e_{12}}^*{S_{e_{11}}^{t_1}}^* \\
&+ \sum_{\overset{l_1,l_3=0, \ l_2'+l_2''\geq r}{l_2\equiv -m_2^{-1}m_1(l_1+1)-{m_2}^{-1}m_3l_3-1 \pmod{r}}}^{r-2} 
S_{e_{11}}^{l_1}S_{e_{12}}S_{e_{22}}^{l_2}S_{e_{23}}{S_{e_{23}}}^*{S_{e_{22}}^{l_2}}^*S_{e_{12}}^*{S_{e_{11}}^{t_1}}^*  
\end{aligned}
$$
Consider the last two sums. We fix $t_1=l_1$ and the corresponding $t_2=-m_2^{-1}m_1(t_1+1) \mod{r}$. In the last sum we then have $l_2\equiv t_2-m_2^{-1}l_3-1 \pmod{r}$. Since $l_2'+l_2''\geq r$, we can only choose the $l_3$ such that $-m_2^{-1}l_3-1 \mod{r}$ varies from $r-t_2$ to $r-1$. Then for a fixed $t_1$, the sum of the two last sums is the following in which $t_2+k$ is calculated modulo $r$: 
$$
\begin{aligned}
&S_{e_{22}}^{t_2}{S_{e_{22}}^{t_2}}^*+\sum_{k=r-t_2}^{r-1} S_{e_{22}}^{t_2+k}S_{e_{23}}{S_{e_{23}}}^*{S_{e_{22}}^{t_2+k}}^*  =S_{e_{22}}^{t_2}{S_{e_{22}}^{t_2}}^*+\sum_{i=0}^{t_2-1} S_{e_{22}}^{i}S_{e_{23}}{S_{e_{23}}}^*{S_{e_{22}}^{i}}^*=P_{v_2} 
\end{aligned}
$$
Then our expression becomes
$$
\begin{aligned}
S_{e_{11}}^{r-1}{S_{e_{11}}^{r-1}}^*+\sum_{\overset{t_1=0}{}}^{r-2} 
S_{e_{11}}^{t_1}S_{e_{13}}S_{e_{13}}^*{S_{e_{11}}^{t_1}}^*  
+\sum_{t_1=0}^{r-2} 
S_{e_{11}}^{t_1}S_{e_{12}}S_{e_{12}}^*{S_{e_{11}}^{t_1}}^* = P_{v_1}
\end{aligned}
$$
and we have proved that the image satisfies \eqref{relation33}. The $*$-homomorphism is injective by a similar argument as in the proof of Proposition \ref{PropGenerators3}. 
\end{proof}
\end{proposition}


In a similar way as for dimension 3 we can immediately construct an explicit isomorphism between $C^*(L_5)^{\rho_{\underline{m}}^r}$ and $C^*(L_5)^{\rho_{\underline{n}}^r}$ preserving the circle action if 
\begin{equation}\label{Criterion1}
\newSS{(m_2,m_{3})}=\newSS{(n_2,n_{3})}, \ \ \newSStilde{(m_1,m_{2})}=\newSStilde{(n_1,n_{2})}
\end{equation}
and 
\begin{equation}\label{Criterion2}
\newSS{(m_1,m_3)}\cup \newSS{(m_1,m_2,m_3)}=\newSS{(n_1,n_3)}\cup \newSS{(n_1,n_2,n_3)}.
\end{equation}
Note that we require $\newSStilde{(m_1,m_{2})}=\newSStilde{(n_1,n_{2})}$ since the condition  $\newSS{(m_1,m_{2})}=\newSS{(n_1,n_{2})}$ might not be sufficient in order to construct the equivariant isomorphism by mapping generators to generators like we did for dimension $3$ since $S_{e_{22}}$ is not normal. Hence we cannot immediately replace Condition \ref{Criterion1} and \ref{Criterion2} by  $\newSS{\underline{m}}=\newSS{\underline{n}}$. We will show that this is indeed possible when $r$ is prime, see Proposition \ref{Prop:IsomorphismDim5}. 
\begin{notation}\label{Notation1}
To be able to describe the isomorphism we will use the following notation. We denote the vertices and edges in $L_5$ by $v_i,i=1,2,3, e_{ij}, 1\leq i\leq j\leq 3$ when considering $C^*(L_5)^{\rho_{\underline{m}}^r}$ and $w_i,i=1,2,3, f_{ij}, 1\leq i\leq j\leq 3$ when considering $C^*(L_5)^{\rho_{\underline{n}}^r}$. The generators described in Proposition \ref{PropGenerators5} for $C^*(L_5)^{\rho_{\underline{m}}^r}$ and $C^*(L_5)^{\rho_{\underline{n}}^r}$ will be denoted by $S_{\alpha_{t_i,t_j}}, S_{\alpha_{t_1,t_2,t_3}}$ and $S_{\beta_{t_i,t_j}}, S_{\beta_{t_1,t_2,t_3}}$ respectively.
\end{notation}


From now on we will consider the case were $r$ is an odd prime. Note that we do not have to deal with $r=2$ since there only exists one quantum lens space in that case. 

We have  
$$
\gcd(m_i-m_k,r)=\begin{cases}
1, & m_i\neq m_j \\
r, & m_i=m_j
\end{cases}
$$
and
$$\newSS{(m_i,m_j)}=\begin{cases}
\{0,1,....,r-1\}, & m_i\neq m_j
\\
\{\overbrace{0,\dots,0}^r\}, & m_i=m_j
\end{cases}
$$ 

In order to prove Proposition \ref{Prop:IsomorphismDim5} we need Lemma \ref{Lemma:5dim} which we illustrate in Example \ref{Ex:5dim}. 
Note that $\newSStilde{(m_1,m_{1})}=\newSStilde{(n_1,n_{1})}$, hence we exclude this case.

Assume that the set of weights takes the form $\underline{m}=(1,m_2,m_3), \underline{n}=(1,n_2,n_3)$ with $m_2,n_2\neq 1$. Let $m:=m_2$ and $n:=n_2$.  For $i\in \{0,1,\dots,r-2\}$ let $t_2,t_2'\in\{0,1,2,\dots,r-1\}$ be such that
$$
1+i+mt_2\equiv 0 \pmod{r}, \ \ \ 1+i+nt_2'\equiv 0 \pmod{r}
$$
Let $X_i:=|\alpha_{i,t_2}|$ and $Y_i:=|\beta_{i,t_2'}|$. Note that $X_{r-1}=Y_{r-1}=r.$
\begin{example}\label{Ex:5dim}
Let $r=7$ then $X_i$ is as given in Table \ref{Table:Xi}. Consider $m=5$ and $n=6$, then the numbers in Table \ref{Table:Xi} that are different are marked in red and the ones that agree in blue. For the two sets of numbers it holds that: 
\begin{enumerate}
    \item $X_i+Y_i\equiv 0 \pmod{r}$,
    \item $X_2=Y_3$, $X_3=Y_2$,
    \item $X_0+r=Y_5, Y_0+r=X_5$,
    \item $X_1+r=Y_4, Y_1+r=X_4$.
\end{enumerate}
\begin{table}[H]
    \centering
    \begin{tabular}{c|c c c c c c c}
         $(m,i)$ & $0$ & $1$ & $2$ & $3$ & $4$ & $5$ & $6$ \\
         \hline
        $2$ & 4 & 8 & 5 & 9 & 6 & 10 & 7\\
        $3$ & 3 & 6 & 9 & 5 & 8 & 11 & 7\\
        $4$ & 6 & 5 & 4 & 10 & 9 & 8 & 7\\
        $5$ & {\color{red} 5} & {\color{red} 3} & {\color{cyan}8} & {\color{cyan}6} & {\color{red} 11} & {\color{red} 9} & {\color{cyan}7}\\
        $6$ & {\color{red} 2} & {\color{red} 4} & {\color{cyan}6} & {\color{cyan}8} & {\color{red} 10} & {\color{red} 12} & {\color{cyan}7}\\
    \end{tabular}
    \caption{The entries is the number $X_i$.}
    \label{Table:Xi}
\end{table}
Note that $(1)$ is not true for all choices of $m$ and $n$ e.g. $m=3$ and $m=6$. On the other hand $(1)$ is true for $m=2$ and $n=3$. Also note that the numbers are the same for $m$ and its multiplicative inverse in $\Zr$.
\end{example}

We will in the proof of Theorem \ref{Prop:IsomorphismDim5} see that if two sets of numbers satisfy similar conditions as in $(2)-(4)$ then we can change the condition $\newSStilde{(m_1,m_{2})}=\newSStilde{(n_1,n_{2})}$ to $\newSS{(m_1,m_{2})}=\newSS{(n_1,n_{2})}$. The crucial observation is that condition $(1)$ is satisfied precisely when $n=(2-m^{1})^{-1} \mod{r}$. Indeed, assume that 
$$
X_i+Y_i=(1+i+t_2)+(1+i+t_2')\equiv 0\pmod{r}.
$$
We have 
$$
\begin{aligned}
(1+i+t_2)+(1+i+t_2')&\equiv
2+2i-(m^{-1}+n^{-1})(1+i) \\
&\equiv (2-m^{-1}-n^{-1})(1+i) \pmod{r}
\end{aligned}
$$
and therefore $n=(2-m^{-1})^{-1}\pmod{r}$.

In Lemma \ref{Lemma:5dim} we will show that condition $(1)$ i.e. $n=(2-m^{-1})^{-1}\mod{r}$ in general implies that the numbers behave as in $(2)-(4)$.

Note moreover that if $n=m^{-1}$ then $\newSStilde{(1,m)}=\newSStilde{(1,n)}$. Indeed assume that $a\equiv -m^{-1}(i+1) \pmod{r}$ and $0<a<r$ then  
$$
X_i=1+i+a.
$$
For $j\in {1,\dots,r-1}$ such that $(1+j)=a$ we have
$$ 
Y_{j}\equiv 1+j-m(j+1)\equiv a-m(-m^{-1}(i+1))\equiv a+i+1 \pmod{r}.
$$
Since $i+1<r$ we obtain $Y_j=X_i$ for this choice of $j$. Hence for any $X_i$ we can find a $Y_j$ such that $X_i=Y_j$.

\begin{lemma}\label{Lemma:5dim}
Let $m\in\mathbb{N}\setminus\{1\}$, $m\neq 2^{-1} \mod{r}$ and $n:=(2-m^{-1})^{-1} \mod{r}$. For each $i\in \{0,1,\dots,r-2\}$, one of the following holds: 
\begin{enumerate}
    \item $X_i=Y_{r-2-i}$ and $Y_i=X_{r-2-i}$
    \item $X_i=Y_{r-2-i}-r$ and $Y_i=X_{r-2-i}-r$
    \item $X_i=Y_{r-2-i}+r$ and $Y_i=X_{r-2-i}+r$
\end{enumerate}
\end{lemma}
\begin{proof}
We have 
$$
X_i=1+i+t_2\equiv 1+i-m^{-1}(i+1) \equiv (1+i)(1-m^{-1}) \pmod{r} 
$$
$$
Y_j=1+j+t_2'\equiv 1+j-(2-m^{-1})(j+1)\equiv (1+j)(m^{-1}-1) \pmod{r}.
$$
Assume $X_j=Y_j$ then 
$$
\begin{aligned}
&(1+i)(1-m_2^{-1})\equiv -(1+j)(1-m_2^{-1}) \pmod{r} \\
&\Rightarrow 1+i\equiv -1-j \pmod{r} \\
&\Rightarrow i+j\equiv -2 \pmod{r}. 
\end{aligned}
$$
Since $0<i+j\leq 2r-4$ we must have $i+j=r-2$ and we obtain
$$
X_i\equiv Y_j \pmod{r} \Leftrightarrow j=r-2-i. 
$$
We have $t_2\equiv -m_2^{-1}(i+1)\pmod{r}$ but since we will no longer calculate modulo $r$ we assume that $a\equiv -m^{-1}(i+1) \pmod{r}$ and $0<a<r$. Then 
$$
X_i=i+1+a.
$$
Since
$$
\begin{aligned}
Y_{r-2-i}&\equiv r-(i+1)+(m^{-1}-2)(r-(i+1)) \pmod{r} \\
&\equiv r-(i+1)+2(i+1)-m^{-1}(i+1) \pmod{r}
\end{aligned}
$$
and $0<a+2(i+1)<3r-1$
we obtain
$$
Y_{r-2-i}=
\begin{cases}
r-(i+1)+a+2(i+1), & a+2(i+1)<r \\
r-(i+1)+a+2(i+1)-r, & r<a+2(i+1)<2r \\
r-(i+1)+a+2(i+1)-2r, & 2r<a+2(i+1) \\
\end{cases}
$$
Then 
$$
Y_{r-2-i}-X_i=
\begin{cases}
r, & a+2(i+1)<r \\
0, & r<a+2(i+1)<2r \\
-r, & 2r<a+2(i+1) \\
\end{cases}
$$
Assume $X_i=Y_{r-2-i}$ as in $(1)$ then we must have $r<a+2(i+1)<2r$. Moreover, since 
$$
\begin{aligned}
m^{-1}(i+1)-2(i+1)&\equiv r-(a+2(i+1)-r)=2r-(a+2(i+1)) \pmod{r}, 
\\
-m^{-1}(r-2-i+1)&\equiv m^{-1}(i+1)\equiv r-a\pmod{r}
\end{aligned}
$$
we obtain 
$$
\begin{aligned}
Y_i&=(1+i)+2r-(a+2(i+1))=2r-a-(i+1)
\\
X_{r-2-i}&=1+(r-2-i)+r-a=2r-a-(i+1). 
\end{aligned}
$$
Hence $Y_i=X_{r-2-i}$ which proves $(1)$. \\
\vspace{0.1cm}
Let now $X_i=Y_{r-2-i}-r$ as in $(2)$ then $a+2(i+1)<r $ and $Y_i= r-a-(i+1)$ hence 
$$
\begin{aligned} 
Y_i-X_{r-2-i}=-r
\end{aligned}
$$
which proves $(2)$. Similarly assume that $X_i=Y_{r-2-i}+r$ then $2r<a+2(i+1)$ hence $Y_i=3r-a-(i+1)$ and $Y_i-X_{r-2-i}=r$ which proves $(3)$. 
\end{proof}

\begin{proposition}\label{Prop:IsomorphismDim5}
Let $r$ be prime and $\underline{m}=(1,m_2,m_3), \underline{n}=(1,n_2,n_3)$ in $\mathbb{N}^3$ with $\gcd(m_i,r)=\gcd(n_i,r)=1$ for $i=2,3$. Assume $\newSS{\underline{m}}=\newSS{\underline{n}}$
then there exists an equivariant isomorphism from $C(L_5)^{\rho_{\underline{m}}^r}$ to $C(L_5)^{\rho_{\underline{n}}^r}$. 
\end{proposition}

\begin{proof}
We will use Notation \ref{Notation1} to describe the two fixed point algebra. The main idea is to change some specific generators by multiplying with $S_{e_{11}}^r$ on $S_{\alpha_{t_1,t_2}}$ and $S_{f_{11}}^r$ on $S_{\beta_{t_1',t_2'}}$. This will be done in such a way that after adding $r$ to some specific numbers in $\newSStilde{(m_1,m_{2})}$ and  $\newSStilde{(n_1,n_{2})}$ (i.e. changing the generator to $S_{e_{11}}^rS_{\alpha_{t_1,t_2}}$ or $S_{f_{11}}^rS_{\beta_{t_1',t_2'}}$ respectively), the two sets equals. 

Changing the generators in this way should be done carefully such that we do not change the relations in the fixed point algebras. First note that $S_{e_{11}}^rS_{\alpha_{t_1,t_2}}$ also satisfies 
$$
\begin{aligned}
&(S_{e_{11}}^rS_{\alpha_{t_1,t_2}})(S_{e_{11}}^rS_{\alpha_{t_1,t_2}})^*\leq P_{v_1}, \\
&(S_{e_{11}}^rS_{\alpha_{t_1,t_2}})^*(S_{e_{11}}^rS_{\alpha_{t_1,t_2}})=P_{v_2}.
\end{aligned}
$$
Therefore we only need to pay attention when summing over all the generators corresponding to paths starting in the first vertex i.e. \eqref{SumOverPaths}. 

The procedure is as follows: Let $i\in\{0,1,\dots,r-1\}$ and assume we change the generator $S_{\alpha_{i,t_2}}$ to $S_{e_{11}}^rS_{\alpha_{i,t_2}}$ then we also want to change the generator $S_{\alpha_{i,l_2,l_3}}$ to $S_{e_{11}}^rS_{\alpha_{i,l_2,l_3}}{S_{e_{33}}^r}^*$. In this way, the result of \eqref{SumOverPaths} becomes: 
$$
\begin{aligned}
S_{e_{11}}^{r-1}{S_{e_{11}}^{r-1}}^*+\sum_{\overset{t_1=0}{}}^{r-2} 
S_{e_{11}}^{t_1}S_{e_{13}}S_{e_{13}}^*{S_{e_{11}}^{t_1}}^*  
+\sum_{t_1=0, t_1\neq i}^{r-2} 
S_{e_{11}}^{t_1}S_{e_{12}}S_{e_{12}}^*{S_{e_{11}}^{t_1}}^* + S_{e_{11}}^rS_{e_{11}}^{i}S_{e_{12}}S_{e_{12}}^*{S_{e_{11}}^{i}}^*{S_{e_{11}}^r}^*
\end{aligned}
$$
We then need to show that if we change $S_{\alpha_{i,t_2}}$ to $S_{e_{11}}^rS_{\alpha_{i,t_2}}$ in $C(L_5)^{\rho_{\underline{m}}^r}$ for some $i$ then we also need to change $S_{\beta_{i,t_2'}}$ to $S_{e_{11}}^rS_{\beta_{i,t_2'}}$ in $C(L_5)^{\rho_{\underline{n}}^r}$ for the same $i$. In this way \eqref{SumOverPaths} becomes the same in the two fixed point algebras. Note that it is always possible to change the corresponding $S_{\alpha_{i,l_2,l_3}}$ ($S_{\beta_{i,l_2',l_3'}}$) since $S_{e_{33}}$ ($S_{f_{33}}$) is normal. 

Assume $m\neq 2^{-1} \mod{r}$. We will first show that if $n_2:=(2-m_2^{-1})^{-1}\mod{r}$ then we can construct an equivariant isomorphism directly. Afterwards we will see how to find an equivariant isomorphism for any $n_2$. 

Assume that $X_i\in \newSStilde{(1,m_{2})}$ with $X_i<r$ and $X_i\notin \newSStilde{(1,n_{2})}$. Then $X_i+r\in \newSStilde{(1,n_{2})}$. For a $j\in\{0,1,\dots,r-2\}$ let $Y_j=X_i+r$. By Lemma \ref{Lemma:5dim} it follows that $j=r-2-i$. Hence we fall into $(3)$ of Lemma \ref{Lemma:5dim} and therefore we have
$$
Y_{r-2-i}=X_i+r \ \text{and} \ X_{r-2-i}=Y_i+r.
$$
We can then change $S_{\alpha_{i,t_2}}$ to $S_{e_{11}}^rS_{\alpha_{i,t_2}}$ in $C(L_5)^{\rho_{\underline{m}}^r}$ and $S_{\beta_{i,t_2'}}$ to $S_{e_{11}}^rS_{\beta_{i,t_2'}}$ in $C(L_5)^{\rho_{\underline{n}}^r}$. Since this happens for the same $i$ we can construct the equivariant isomorphism as described earlier. On the other hand assume that $X_i\in \newSStilde{(1,m_{2})}$ with $X_i>r$ and $X_i\notin \newSStilde{(1,n_{2})}$. Then $X_i-r\in \newSStilde{(1,n_{2})}$. It follows by Lemma \ref{Lemma:5dim} case $(3)$ that 
$$
X_i=Y_{r-2-i}+r \ \text{and} \ Y_i=X_{r-2-i}+r.
$$
We can then change $S_{\alpha_{r-2-i,t_2}}$ to $S_{e_{11}}^rS_{\alpha_{r-2-i,t_2}}$ in $C(L_5)^{\rho_{\underline{m}}^r}$ and $S_{\beta_{r-2-i,t_2'}}$ to $S_{e_{11}}^rS_{\beta_{r-2-i,t_2'}}$ in $C(L_5)^{\rho_{\underline{n}}^r}$ for the same index $r-2-i$. 

To show that there exists an equivariant isomorphism for an arbitrary choice of $n_2$ we first let $m_2=2^{-1} \mod{r}$. Since $\newSStilde{(1,m_{2})}=\newSStilde{(1,m_{2}^{-1})}$ we have $(1,2^{-1},m_3)\simeq_{\gamma}(1,2,n_3)$ (see Notation \ref{not:equiviso}). Now we can consider $m=2$ and use the previous result to conclude $(1,2,m_3)
{\color{black}\simeq_{\gamma}} (1,(2-2^{-1})^{-1},n_3)$ and so on. 

On the other hand let $m_2=r-1$ then $m_2^{-1}=r-1$. By the above we have $(1,m_2,m_3)\simeq_{\gamma}(1,n_2,n_3)$ where
$$
n_2\equiv (2-m_2^{-1})^{-1}\equiv (2-(r-1))^{-1}\equiv 3^{-1} \pmod{r}.
$$
Moreover we have $(1,3^{-1},m_3)\simeq_{\gamma}(1,3,n_3)$. By the previous result we obtain $(1,3,m_3)\simeq_{\gamma}(1,(2-3^{-1})^{-1},n_3)$. 
This process is illustrated below where a blue $\simeq_{\gamma}$ indicates the "easy" isomorphism i.e. when $\newSStilde{(1,m_{2})}=\newSStilde{(1,n_2)}$ and a red one indicates the isomorphism constructed in this proof.
$$
\begin{aligned}
(1,2^{-1},\cdot)&{\color{blue}\simeq_{\gamma}} (1,2,\cdot)
{\color{red}\simeq_{\gamma}} (1,(2-2^{-1})^{-1},\cdot)
{\color{blue}\simeq_{\gamma}} (1,(2-2^{-1}),\cdot)
{\color{red}\simeq_{\gamma}} (1,(2-(2-2^{-1})^{-1})^{-1},\cdot)
\\
&{\color{blue}\simeq_{\gamma}} (1,2-(2-2^{-1})^{-1},\cdot){\color{red}\simeq_{\gamma}} \cdots\cdots {\color{blue}\simeq_{\gamma}} (1,(2-3^{-1})^{-1},\cdot)
{\color{red}\simeq_{\gamma}}
(1,3,\cdot) \\
&{\color{blue}\simeq_{\gamma}} (1,3^{-1},\cdot)
{\color{red}\simeq_{\gamma}} (1,r-1,\cdot).
\end{aligned}
$$
Continuing this process and combining the isomorphisms we end up with an equivariant isomorphism between any two set of weigths $(1,m_2,m_3)$ and $(1,n_2,n_3)$ satisfying the assumptions of the theorem. 
\end{proof}

\begin{proposition}\label{prop:explicit_iso}
Let $r$ be a prime number and $\underline{m},\underline{n}\in \Nb^3$  with $\gcd(m_i,r)=\gcd(n_i,r)=1$ for $i=1,2,3$ for which
\begin{equation}\label{3DifferentWeights}
|\{m_1,m_2,m_3\}|=|\{n_1,n_2,n_3\}|=3,
\end{equation}
then we obtain $\newSS{\underline{m}}=\newSS{\underline{n}}$. 
\begin{proof}
First note that since $r$ is prime we have 
$$\begin{aligned}
\newSS{(m_i,m_{i+1})}=\newSS{(n_i,n_{i+1})}, i=1,2, \ \ \ 
\newSS{(m_1,m_3)}=\newSS{(n_1,n_3)}
\end{aligned}
$$
Hence we only need to show that $\newSS{(m_1,m_2,m_3)}=\newSS{(n_1,n_2,n_3)}$. 

Let the numbers $t_1$ and $t_1'$ depend on $t_2',t_2''\in \{0,1,\dots,r-1\}$ respectively as in \eqref{tandk}. Then
$$
\begin{aligned}
t_1+t_2+t_3+2&\equiv t_1+t_3+t_2'+t_2''+2 \pmod{r} \\
&\equiv -m_2m_1^{-1}t_2'-1-m_2m_3^{-1}(t_2''+1) + t_2'+t_2''+2 \pmod{r}  \\
&\equiv t_2'(1-m_2m_1^{-1})+(t_2''+1)(1-m_2m_3^{-1}) \pmod{r}. 
\end{aligned}
$$
Then the numbers in $\newSS{(m_1,m_2,m_3)}$ will be the $t_1+t_3+t_2'+t_2''+2$ for which $t_2'+t_2''\geq r$. Here it is an advantage to think about the numbers $t_1+t_3+t_2'+t_2''+2$ as the entries of a $r\times r$ matrix, where the rows are indexed by $t_2'$ and the columns by $t_2''$. Then the entries below the antidiagonal will precisely be the elements in $\newSS{(m_1,m_2,m_3)}$, hence it consists of $\frac{r(r-1)}{2}$ numbers. 

Since $(1-m_2m_3^{-1})$ is a unit, we choose to consider the numbers $(1-m_2m_3^{-1})^{-1}(t-1+t_2+t_3+2)$ instead of $t_1+t_2+t_3+2$. If we can show that the multiset 
$$\{(1-m_2m_3^{-1})^{-1}(t_1+t_2+t_3+2), t_1+t_2+t_3+2\in \newSS{(m_1,m_2,m_3)}\}$$
is the same for any set of weights under the given condition, then $\newSS{(m_1,m_2,m_3)}$ would also be the same. Let $k:=(1-m_2m_1^{-1})(1-m_2m_3^{-1})^{-1}$, then we have to consider the numbers $t_2'k+1+t_2''$. For each $x=1,\dots,r-1$, if $t_2''=r-x$ we can choose $t_2'\in \{x,\dots,r-1\}$, hence we have to consider the following numbers: 
$$
\{t_2'k+1-x \mod{r}\lvert \hspace{0.2cm} x\in \{1,\dots,r-1\}, t_2'\in \{x,\dots,r-1\}\}. 
$$
To get a better idea of the numbers we consider Table \ref{TheNumbers} which indicates the number $t_2'k+1-x$ for any combination of $x,t_2'\in \{1,\dots,r-1\}$. The numbers in $\newSS{(m_1,m_2,m_3)}$ are precisely the ones, calculated modulo $r$, marked with blue in Table \ref{TheNumbers}. 
\begin{table}[H]
    \centering
    \begin{tabular}{c|c c c c c}
         $(x,t_2')$ & $1$ & $2$ & $3$ & $\cdots$ & r-1  \\
         \hline
        $r-1$ & $k-(r-2)$ & $2k-(r-2)$ & $3k-(r-2)$ &$\cdots$ & \color{cyan}{$(r-1)k-(r-2)$}\\
        $\vdots$ &   $\vdots$  &  $\vdots$  & &  \color{cyan}{$\vdots$}   \\
        $3$ & $k-2$ & $2k-2$& \color{cyan}{$3k-2$} & \color{cyan}{$\cdots$} & \color{cyan}{$(r-1)k-2$} 
        \\
        $2$ & $k-1$ & \color{cyan}{$2k-1$} & \color{cyan}{$3k-1$} & \color{cyan}{$\cdots$} & \color{cyan}{$(r-1)k-1$}  \\
        $1$ & \color{cyan}{$k$} & \color{cyan}{$2k$} & \color{cyan}{$3k$} &  \color{cyan}{$\cdots$} & \color{cyan}{$(r-1)k$} 
    \end{tabular}
    \caption{The entries indicates the number $t_2'k+1-x$.}
    \label{TheNumbers}
\end{table}
It is clear that the numbers $0,1,\dots,r-1$ appears precisely one time in each row and column of Table \ref{TheNumbers}. Moreover we have that a number only appears ones on each antidiagonal and diagonal. Indeed, first note that the elements on the same antidiagonal takes the form $yk-(y-b)$ for a fixed $b\in \Nb$. Assume now for $y\neq z$ that $yk-(y-b)\equiv zk-(z-b) \Mod{r}$, then $y(k-1)\equiv z(k-1)\Mod{r}$ but since $r$ is prime we obtain $y\equiv z \Mod{r}$ i.e. $y=z$. We can show a similar results for each diagonal by noting that elements on the same diagonal takes the form  $yk+(y-b)$ for a $b\in \Nb$.
Since each number appears precisely ones in each row and column and a number only can appears one time in each diagonal and antidiagonal, each of the numbers $0,\dots,r-1$ appears precisely $\frac{r-1}{2}$-times above and below the antidiagonal. Then $\newSS{(m_1,m_2,m_3)}$ consists of each of the numbers $0,\dots,r-1$, $\frac{r-1}{2}$-times. Hence $\newSS{(m_1,m_2,m_3)}$ is the same for any choice of weights under our assumption. 
\end{proof}

\end{proposition}

\section{Dimension 5, $r$ prime, necessity}

 In case $d=5$ we have seen this far that when $\mm=\alpha \nn$ or when $|\{m_1,m_2,m_3\}|=3=|\{n_1,n_2,n_3\}|$, the corresponding quantum lens spaces are equivariantly isomorphic. We now set out to prove that when neither of these conditions hold, the quantum lens spaces fail to be equivariantly isomorphic, by showing that their dimension quadruples are not isomorphic. 
 
 Our strategy for doing so is to first assume, as we may by Lemma \ref{Lemma:simpleiso}, that $m_1=n_3=1$. If $\DQ(r;\mm)\simeq  \DQ(r;\nn)$ via $H$, it must be of the form
 \[
H=\begin{bmatrix}
\perm_{n_1}&0&0\\
Y_{21}&\sh^{\ell_2}\perm_{n_2/m_2}&0\\
Y_{31}&Y_{32}&\sh^{\ell_3}\perm_{1/m_3}
\end{bmatrix} 
\]
where we can extract information about $\Delta(Y_{21})$ and $\Delta(Y_{32})$ by our previous analysis of the case $d=3$. We have that
\[
H
\begin{bmatrix}
\sh&0&0\\
-\sh^{m_2}&\sh^{m_2}&0\\
-\sh^{m_3}&-\sh^{m_3}&\sh^{m_3}
\end{bmatrix}
=
\begin{bmatrix}
\sh^{n_1}&0&0\\
-\sh^{n_2}&\sh^{n_2}&0\\
-\sh&-\sh&\sh
\end{bmatrix}
H,
\]
which at the  $31$ entry becomes
\[
Y_{31}\sh-Y_{32}\sh^{m_2}-\sh^{\ell_3}\perm_{1/m_3}\sh^{m_3}=
-\sh \perm_{n_1}-\sh Y_{21} +\sh Y_{31}
\]
or
\begin{equation}\label{masterfive}
\sh Y_{31}-Y_{31}\sh
=\sh \perm_{n_1}+\sh Y_{21} -
Y_{32}\sh^{m_2}-\sh^{\ell_3+1}\perm_{1/m_3}
\end{equation}
Using Lemma \ref{SYYSZ}(ii) we infer that
\[
\Delta(\sh \perm_{n_1}+\sh Y_{21} -
Y_{32}\sh^{m_2}-\sh^{\ell_3+1}\perm_{1/m_3})=0
\]
from which we shall extract the necessary information.
\\

The following result is  elementary.

\begin{lemma}
With $\gamma:\ZZ/r\times\ZZ/r\to \NN_0$ defined as
\[
\gamma(\alpha,\beta)=\sum_{j=0}^{r-1}\iverson{\alpha j+\beta\leq j}
\]
we get
\[
\gamma(\alpha,\beta)=\begin{cases}r&\alpha=0,\beta=0\\ r-\beta \mod r&\alpha=0, \beta\not=0\\r&\alpha=1,\beta=0\\\beta \mod r&\alpha=1,\beta\not =0\\\dfrac{r+1}{2}&\text{otherwise}\end{cases}
\]
\end{lemma}

\begin{proposition}\label{deltadetails} 
When $d=5$ and $r$ is a prime, and $H$ is the unique matrix implementing $\DQ(r;\mm)\simeq \DQ(r;\nn)$ with $y_{i}=0$ as in Lemma \ref{describeH} then if 
$m_1\not= m_2$ we have
\[
\Delta(Y_{21})=\left(\gamma\left(\frac{m_1-n_2}{n_2p_M},\frac{i}{n_2p_M}\right)-\gamma\left(\frac{m_1-n_2}{n_2p_N},\frac{i}{n_2p_N}\right)\right)
\]
with $p_N={n_1/n_2-1} $ and $p_M=m_1/m_2-1$, and when $m_1=m_2$ we have
\[
\Delta(Y_{21})=\oo
\]
\end{proposition}
\begin{proof}
For both $p\in\{p_M,p_N\}$ we get
\begin{eqnarray*}
\sum_{j=0}^{r-1}\iverson{\frac{(j+i)/n_2+j/m_1}p\leq \frac{j}{m_1}}
&=&\sum_{j=0}^{r-1}\iverson{j\frac{m_1-n_2}{m_1n_2 p}+i\frac1{n_2p}\leq {j}\frac1{m_1}}\\
&=&\sum_{j=0}^{r-1}\iverson{(j/m_1)\frac{m_1-n_2}{n_2 p}+i\frac1{n_2p}\leq j/m_1}\\
&=&\sum_{j=0}^{r-1}\iverson{j\frac{m_1-n_2}{n_2 p}+i\frac1{n_2p}\leq j}\\
&=&\gamma\left(\frac{m_1-n_2}{n_2p},\frac{i}{n_2p}\right)
\end{eqnarray*}
and collect the terms as in Lemma \ref{describeH}.
\end{proof}

Note that we have only described the $\Delta$-vector in the special case where the $\ell_i$ and $y_{i0}$ all vanish, and even in that case, the entries of $\Delta(Y_{21})$ vary in a rather complicated way with the parameters. Fortunately, even though we need the entries of $\Delta$ to match up precisely to invoke Lemma \ref{SYYSZ}(ii) in the positive direction, we can get by with much less for our negative purposes. We formalize this by passing to $\overline{\Delta(Y)}$ defined by
\[
\overline{\Delta(Y)}_i:=\Delta(Y)_i\mod r
\]
and the ad hoc equivalence relation defined below.

\begin{definition}
When $r$ is fixed, we define $\sim$ as the coarsest equivalence relation of $(\Zr)^r$ so that
\[
(x_0,\dots,x_{r-1})\sim (x_0+1,\dots,x_{r-1}+1)\qquad (x_0,\dots,x_{r-1})\sim (x_1,\dots,x_{r-1},x_0)
\]
\end{definition}

\begin{lemma}\label{siminv}
Suppose $d=3$ and $H,\widetilde{H}$ both induce an isomorphism $\DQ(r;\mm)\simeq \DQ(r;\nn)$. Then 
\[
\overline{\Delta}(Y_{21}) \sim \overline{\Delta}(\widetilde{Y}_{21})
\]
\end{lemma}
\begin{proof}
If both $H$ and $\widetilde{H}$ have the form analyzed in Lemma \ref{describeH}, we get that 
\[
\Delta{(Y_{21})}-\Delta{(\widetilde{Y}_{21})}=\left(\sum_{j=0}^{r-1} y_{j/n_2-(i+j)/m_1}\right)
\]
which is either constant
$
\sum_{i=0}^{r-1}y_i
$
or varies with each $ry_i$ appearing exactly once, depending on whether or not $m_1=n_2$. In either case, $\overline{\Delta}(Y_{21})\sim\overline{\Delta}(\widetilde{Y}_{21})$.

When $H$ or $\widetilde{H}$ are given with $\ell_1\not=0$ or $\ell_2\not=0$, we need to adjust by $\Delta(X_{21})$ with $X_{21}$ as constructed in the proof of  Lemma \ref{describeH}. Computations similar to those for $Y_{21}$ show that $\Delta(X_{21})\sim \oo$.
\end{proof}

Because of Lemma \ref{siminv}, it makes sense to talk about $\overline{\Delta}(\mm,\nn)$ up to $\sim$-equivalence.


\begin{proposition}\label{deltadetailsplus}
For any choice of $\mm$ and $\nn$, we have
\[
\overline{\Delta}(\mm,\nn)\sim {\textsf x}_\gamma=(0,\gamma,2\gamma,\dots,(r-1)\gamma)
\]
for a unique choice of $\gamma$ given by the table
%
\begin{center}
\begin{tabular}{|c|c|c|c|}\hline
$|\{m_1,m_2,n_1,n_2\}|$&Subcase&Identities&$\gamma$\\\hline\hline
$4$&&&$0$\\\hline
\multirow{4}{*}{$3$}&(a)&$m_1=n_1$&$1/n_2p_N$\\
&(b)&$m_1=n_2$&$1/n_2p_M-1/n_2p_N$\\
&(c)&$m_2=n_1$&$0$\\
&(d)&$m_2=n_2$&$1/n_2p_M$\\\hline
\multirow{2}{*}{$2$}&(a)&$m_1=n_1,m_2=n_2$&$0$\\
&(b)&$m_1=n_2,m_2=n_1$&$1/n_2p_M-1/n_2p_N$\\\hline
\end{tabular}
\end{center}
\end{proposition}
\begin{proof}
We note that $\overline{\gamma}:\Zr\times\Zr\to\Zr$ defined by $ \overline{\gamma}(\alpha,\beta):={\gamma}(\alpha,\beta)\mod r$ is just
\[
\overline{\gamma}(\alpha,\beta)=\begin{cases}-\beta&\alpha=0\\\beta&\alpha=1\\\frac{r+1}2&\text{otherwise}\end{cases}
\]
To determine which of these three cases are relevant for the difference in Proposition \ref{deltadetails} we note that
\[
\frac{m_1-n_2}{n_2p_M}=0\Longleftrightarrow \frac{m_1-n_2}{n_2p_N}=0\Longleftrightarrow m_1=n_2
\]
and that
\begin{eqnarray*}
\frac{m_1-n_2}{n_2p_M}=1&\Longleftrightarrow m_1=n_1\\
\frac{m_1-n_2}{n_2p_N}=1&\Longleftrightarrow m_2=n_2.
\end{eqnarray*}
Combining these observations show all claims but the one at (2)(a), where we also need to note that in this case, $p_M=p_N$.
\end{proof}

\begin{proposition}\label{Prop:DimQ}
Suppose  $d=5$, $r$ is a prime and $\DQ(r;\mm)\simeq \DQ(r;\nn)$. When
\[
|\{m_1,m_2,m_3\}|<3
\]
there exists an $\alpha\in \Zru$ so that $\nn=\alpha \mm$.
\end{proposition}
\begin{proof}
We may pass to the situation  $m_1=n_3=1$ as discussed above. In this case, we must have $\alpha=n_1$, so our aim will be to prove 
\[
n_2=n_1m_2\qquad 1=n_1m_3.
\]
We argue from \eqref{masterfive}. Since $\overline{\Delta}(P_\alpha)\sim \oo$ for any $\alpha\in\Zru$, we always know that 
\[
\overline{\Delta}(Y_{21})\sim \overline{\Delta}(Y_{32})
\]
or in other words
\[
\overline{\Delta}((1,m_2),(n_1,n_2))\sim \overline{\Delta}((m_2,m_3),(n_2,1))
\]
We work casewise over the various ways that the $m_i$ can agree.

\noindent \underline{{$0^\circ$}: $m_2=m_3=1$}\\

We have $\gcd(m_2-m_1,r)=\gcd(m_3-m_2,r)=r$ and consequently also $\gcd(n_2-n_1,r)=\gcd(n_3-n_2,r)=r$ as seen in Theorem \ref{Thm:IdealLattice}. We conclude that $n_1=n_2=n_3=1$ as desired.

\noindent \underline{{$1^\circ$}: $m_2=1\not=m_3 $}

Our task is to establish 
\[
n_1=n_2=1/m_3.
\]
Again we know from Theorem \ref{Thm:IdealLattice}  that $n_1=n_2$ so it suffices to show that $n_1=1/m_3$, knowing that 
\[
\overline{\Delta}((1,m_3),(n_2,1))=\oo.
\]
Since we are in either Case (3)(b) or (2)(b) of Proposition  \ref{deltadetails}, we conclude that $p_M=p_N$, i.e. $1/m_3=n_2/1$ as desired.\\

\noindent \underline{{$2^\circ$}: $m_2=m_3\not=1 $} \\
Exactly as case $1^\circ$.
\\

\noindent \underline{{$3^\circ$}: $m_3=1\not=m_2$}

This time our goal is to establish
\[
n_1=1,m_2=n_2,
\]
knowing that 
\begin{equation}\label{blacktriangle}
\overline{\Delta}((1,m_2),(n_1,n_2))\sim \overline{\Delta}((m_2,1),(n_2,1))
\end{equation}

This time the possible cases from Proposition  \ref{deltadetailsplus}  are 
\[
(2)(a),(3)(a),(3)(c),(3)(d),(4)
\]
to the left, and
\[
(2)(a),(3)(d)
\]
to the right. 

If we had been in case (3)(d)  to the right, the vectors in \eqref{blacktriangle} would not vanish, and thus the possible cases to the left are (3)(a) and (3)(d). Having both cases be (3)(d) is not consistent, and in the case with (3)(a) to the left and (3)(d) to the right, we infer from  Proposition  \ref{deltadetailsplus} that $m_2=n_2$, which is a contradiction.

Consequently we must have case (2)(a) to the right, and both  vectors in \eqref{blacktriangle} vanish. The possible cases to the left are thus (2)(a),(3)(c),(4), of which only (2)(a) is consistent with the choice of (2)(a) to the right, proving the claim.
%
%
%
%
\end{proof}

\begin{theorem}\label{mainresult5}
Let $r\in\Nb$ be a prime,  and consider $\underline{m}=(m_1,m_2,m_3)$, $\underline{n}=(n_1,n_2,n_3)$ in $((\ZZ/r)^\times)^3$. Then the following are equivalent: 
\begin{enumerate}[(i)]
\item $(\QLS[5],\gamma)\simeq (\QLSalt[5]),\gamma)$. 
\item $\DQp(r;\mm)\simeq \DQp(r;\nn)$
\item $(\QLS[5]\otimes\KKK,\gamma\otimes\id)\simeq (\QLSalt[5]\otimes\KKK,\gamma\otimes\id)$.
\item $\DQ(r;\mm)\simeq \DQ(r;\nn)$
\item $\newSS{\mm}=\newSS{\nn}$ 
\item $|\{m_1,m_2,m_3\}|=3=|\{n_1,n_2,n_3\}|$ or $\mm=\alpha \nn$ for some $\alpha \in \Zru$
\end{enumerate}
\end{theorem}
\begin{proof}
We get (i)$\Longrightarrow$(ii) and (iii)$\Longrightarrow$(iv) by  Theorem \ref{Thm:DQ0}, and $(i)\Longrightarrow$(iii) is clear. Proposition \ref{Prop:DimQ} shows (iv)$\Longrightarrow$(vi), and we get (vi)$\Longrightarrow$(v) by Proposition \ref{Prop:IsomorphismDim5} . Proposition \ref{prop:explicit_iso} and Lemma \ref{Lemma:simpleiso} closes the circle by showing (v)$\Longrightarrow$(i). 
\end{proof}

\begin{remark}\label{zerocols}
We note that whenever $(\QLS[5],\gamma)\simeq (\QLSalt[5]),\gamma)$ with $r$ a prime, we can induce the isomorphism  $\DQ(r;\mm)\simeq \DQ(r;\nn)$ with an $H$ in Proposition \ref{charH} satisfying the additional conditions that $\ell_c=0$ for all $1\leq c\leq 3$ and that $Y_{ab}(i,0)=0$ for all $i$ and all $1\leq b<a\leq 3$.

To see this, we argue from (v). In the case $\mm=\alpha\nn$ we can take
\[
H=\begin{bmatrix}\perm_\alpha &0&0\\
0&\perm_\alpha&0\\
0&0&\perm_\alpha
\end{bmatrix},
\]
and in the other case, we may assume that $m_1=1=n_3$.  Using Lemma \ref{describeH}, we see that we can choose $Y_{21}$ and $Y_{32}$ so that $H$ will intertwine $\myA[r;\mm]$ and $\myA[r;\nn]$ at the 21 and 32 block matrix entries, and  with the first column vanishing. The options of subcases of  Proposition  \ref{deltadetailsplus} then become
\begin{center}
\begin{tabular}{|c|c|c|}\hline
$|\{1,m_2,m_3,n_1,n_2\}|$&Left&Right\\\hline
5&4&4\\\hline
\multirow{4}{*}{$4$}&3(d)&3(a)\\
&4&4\\
&3(c)&4\\
&4&3(c)\\\hline
\multirow{2}{*}{$3$}&3(c)&3(c)\\
&3(d)&3(a)\\\hline
\end{tabular}
\end{center}
and we see by inspecting Lemma \ref{deltadetails} in all these cases that not only do the $\overline{\Delta}$ match up up to $\sim$, but in fact that the ${\Delta}$ match up up to identity. Thus we can choose $Y_{31}$ with the first column vanishing.
\end{remark}
\section{Further cases}

In this final section of the paper, we will compare the statements 
\begin{enumerate}[(I)]
\item $(\QLS,\gamma)\simeq (\QLSalt,\gamma)$. 
\item $\DQp(r;\mm)\simeq \DQp(r;\nn)$
\item $(\QLS\otimes\KKK,\gamma\otimes\id)\simeq (\QLSalt\otimes\KKK,\gamma\otimes\id)$.
\item $\DQ(r;\mm)\simeq \DQ(r;\nn)$
\item $\newSS{\mm}=\newSS{\nn}$ 
\item ($|\{m_1,\dots,m_{k+1}\}|=k+1=|\{n_1,\dots,n_{k+1}\}|$ and  $\gcd(m_{i+1}-m_i,r)=\gcd(n_{i+1}-n_i,r)$), or $\mm=\alpha \nn$ for some $\alpha \in \Zru$
\item We can induce the isomorphism  $\DQ(r;\mm)\simeq \DQ(r;\nn)$ with an $H$ as in Proposition \ref{charH} satisfying the additional conditions that $\ell_c=0$ for all $1\leq c\leq k+1$ and that $Y_{ab}(i,0)=0$ for all $i$ and all $1\leq b<a\leq k+1$.
\end{enumerate}
for general $r,\mm,$ and $\nn$. 

Note that conditions (I)--(V) exactly parallel the conditions of Theorem \ref{mainresult5}. Condition (VI) combines (vi) from  Theorem \ref{mainresult5}
with  Theorem \ref{Thm:IdealLattice}. This becomes necessary when $r$ is composite, since the $\gcd$ is no longer the same for any differing pair of units, as it were for prime $r$.

Condition (VII) is motivated in computability issues for the purposes of comparing the various invariants from (II), (IV), (V), (VI), (VII) by computer experiments. All of these invariants can be computed and compared algorithmically, but whereas it is straightforward to do so for (V) and (VI) -- in fact  the computation of $\newSS{-}$ can be done very efficiently by matrix multiplication -- the necessary solution of linear systems over $\ZZ$ to decide (II) and (IV) are quite time-consuming. 

The fact that we need to allow for general $\ell_c$ in Proposition \ref{charH}  is a major problem in this regard, because the non-linearity of the impact of the $\ell_c$ makes it impossible to solve for these values. Consequently, all possibilities must be checked before one can discard the possibility of $\DQ(r;\mm)\simeq \DQ(r;\nn)$ or $\DQp(r;\mm)\simeq \DQp(r;\nn)$ to hold. It is worth noting that we proved for $d\leq 3$ and for $d\leq 5$ and $r$ a prime that whenever an isomorphism exists, there is one with $\ell_c=0$ for all $c$, but we do not know how to show this in general. The most basic issue here is that it is no longer true at $d=5$ that all possible choices of $\ell_c$ can be realized by automorphisms and then disregarded as we did for $d=3$. 

We introduce (VII)  as a strengthening of condition (IV), where we not only assume that all the $\ell_c$ vanish, but also that the condition $H[p]=[p]$ is satisfied in the obvious way of having all first columns of all subdiagonal blocks vanish. Checking this can be done substantially faster, because of the fact that Lemma \ref{SYYSZ}(i)  shows that the linear system determining the existence or non-existence of $Y$ has at most one solution. As noted in Remark \ref{zerocols}, we can prove that (IV) (and hence (II)) implies this stronger condition in the cases we have solved theoretically. We have  implemented tests for all conditions in Maple, but carrying out these tests for (II) and (IV) is prohibitively costly for $r$ or $d$ much further beyond the instances we understand theoretically.

We summarize our work this far in

\begin{theorem}
Conditions (I)--(VII) are equivalent when $d=3$ or when $d=5$ and $r$ is a prime.
\end{theorem}

Comparing (V)--(VII) by computer experiments is possible for several choices of $r$ and $d$, and we have done so systematically for
\begin{itemize}
\item $d=5$ for $r\leq 31$
\item $r=5$ for $d\leq 13$
\item $r=8$ for $d\leq 11$
\end{itemize}
(recall that $r\in \{3,4,6,12\}$ are covered by Corollary \ref{easy34612}, so the two choices of $r$ are the smallest interesting prime and composite numbers, respectively.) Since no differences were found at $d=5$ for composite $r$, we venture:

\begin{conjecture}
Conditions (I)--(VII) are equivalent when $d\leq 5$.
\end{conjecture}

In fact, we find it likely that the remaining cases could be proved along the lines of the work already presented. But there are several outstanding technical issues which we are somewhat daunted by.

Beyond $d=7$, several new phenomena occur:

\begin{example}
Consider the case  $r=5$. With 
\[
\mm=(1,2,3,1)\qquad \nn=(1,3,2,1)
\]
we have that (VII) holds but (VI) fails. With 
\[
\mm=(1,2,3,4)\qquad \nn=(1,2,4,3)
\]
we have that (VII) holds but (V) fails. With
\[
\mm=(1,3,4,1,2,3)\qquad \nn=(1,4,3,1,2,4)
\]
we have that (II) holds but (VII) fails.
\end{example}

The two first examples were found by the brute force search described above. The first example indicates that it is indeed possible to have $(\QLS[7],\gamma)\simeq (\QLSalt[7],\gamma)$ nontrivially, even though not all four entries of the $\mm$ and $\nn$  are mutually different (but we do not know that (II)$\Longrightarrow$ (I) in order to prove so). The second example is less of a surprise, because  we know from \cite[Theorems 7.8 and 7.9]{segrerapws:gcgcfg} it is possible for the adjacency matrices defining the graphs used in \cite{jhhws:qlsga} to differ even though their graph $C^*$-algebras (i.e., the quantum lens spaces) are isomorphic. For the concrete pair found by our program, the number of admissible paths from the top to the bottom is 40 and 45, respectively, so  necessarily the $\newSS{-}$ differ because there is not the same number of elements in the corresponding multiset.

The last example was found by noting a lack of symmetry amongst the pairs of tuples at $d=11$ that agree in the sense of (VII), and testing that the concrete choice shown here in fact has a solution implemented by $H$ with $\ell_c=0$ for all $c$ and satisfying \eqref{H1eq}, but cannot be found with all first columns of all $Y_{ab}$ vanishing. This establishes that the shortcut we were using is not generally feasible at $d\geq 11$, but we tend to believe that it works at lower dimensions. 

Combining  our experiments with a bit of qualified guesswork, we believe that it is possible to formulate a criterion describing exactly, at any dimension, when $(\QLSr[5],\gamma)$ $\simeq (\QLSralt[5],\gamma)$ nontrivially (i.e.~with $\nn\not=\alpha \mm$), describing the  patterns allowing this phenomenon as a regular language (\cite{sck:rennfa},\cite{js:are}). For this, we note that at $d=7$, the only nontrivial pairs are chosen in a row inside Table \ref{tab:nontrivialsol}, up to trivial isomorphism. We denote the four different types by 0,1,2, and 3. 
\begin {table}[H]
\caption{Nontrivial solutions} \label{tab:nontrivialsol} 
\begin{center}
\begin{tabular}{|c|r|}
\hline
\texttt Type & \\\hline
\texttt 0&$(1,2,3,4),(1,2,4,3),(1,3,2,4),(1, 3,4,2),(1,4,2,3),(1,4,3,2)$\\\hline
\texttt 1&$(1,2,3,1),(1,3,2,1)$\\\hline
\texttt 2&$(1,3,4,1),(1,4,3,1)$\\\hline
\texttt 3&$(1,2,4,1),(1,4,2,1)$\\\hline
\end{tabular}
\end{center}
\end{table}
At higher dimensions, we observe that any nontrivial solution at $d=5+2\ell$ can be described by a string of $\ell$ numbers, corresponding to the types above, defining how these patterns are combined. For instance, when the program finds that condition  (VII) is satisfied at 
\[
\mm=(1,4,2,3,4,1,3)\qquad \nn=(1,4,3,2,4,1,2)
\]
we transform by a trivial isomorphism, four entries at a time, to elements found in the table as follows:
\begin{center}
\begin{tabular}{ccccccc|c|ccccccc}
1&4&2&3&4&1&3& \text{Type} &1&4&3&2&4&1&2\\
\hline
1&4&2&3&&&&\texttt 0&1&4&3&2&&&\\
&1&3&2&1&&&\texttt 1&&1&2&3&1&&\\
&&1&4&2&3&&\texttt 0&&&1&4&3&2&\\
&&&1&3&2&1&\texttt 1&&&&1&2&3&1\\
\end{tabular}
\end{center}
The table should be read as follows: first we consider the first four entries $1 4 2 3$ which is of type {\texttt 0} in Table \ref{tab:nontrivialsol}. Then we consider the next four entries $4 2 3 4$, to get it into a form as in Table \ref{tab:nontrivialsol} (i.e. the first entry must be 1) we multiply through with $4$ (the inverse of $4$). Hence we obtain $1 3 2 1$ which is of Type {\texttt 1}. 
  
We then say that the nontrivial solution is given by the pattern \texttt{0101}. It is easy to see that the pattern \texttt{0$\cdots$0} corresponds to 24 choices of parameters, and that any other pattern has $0$ or $8$ choices. Some patterns such as \texttt{11} are not realizable by parameters, whereas others such as \texttt{202} are, but seem not to generate equivariantly isomorphic quantum lens spaces. We venture the following guess.

\newcommand{\ts}[1]{\text{\tt #1}}
\begin{conjecture}
Any nontrivial isomorphism $(\QLSr[5],\gamma)\simeq (\QLSralt[5],\gamma)$ at $d\geq 7$ has a pattern in the regular language described by
\[
\ts{0}^*\vee \ts{0}^*\ts{1}\vee \ts{0}^*(\ts{10}(\ts{00})^*)^*\ts{0}^*\vee \ts{0}^*\ts{2}\ts{0}^*\vee  \ts{0}^*\ts{3}\ts{0}^*
\]
and any such pattern is realized by nontrivial isomorphisms with 24 choices when the word   is of the form $\ts{0}\ts{0}^*$ and 8 at all other words in the language.
\end{conjecture}

The notation used above uses the \emph{Kleene star} ``$*$'' to signify that a word can be repeated any finite number of times (including zero times), an implicit symbol for concatenation, and ``$\vee$'' for alternative. As is customary in theoretical computer science, the  operations take priority of evaluation in the order  stated.

A similar behavior can be observed at $r=8$. At $d=7$ we have nontrivial isomorphisms
\begin{center}
\begin{tabular}{|c|r|}\hline
\text{Type} & \\
\hline
&$(1,3,5,7),(1,7,5,3)$\\\cline{2-2}
\texttt 0&$(1,3,7,5),(1,7,3,5)$\\\cline{2-2}
&$(1,5,3,7),(1,5,7,3)$\\\hline
\texttt 1&$(1,3,7,1),(1,7,3,1)$\\\hline
\end{tabular}
\end{center}
where the pattern \texttt{0$\cdots$0} corresponds to three different groups of 8 choices of parameters, and any other pattern has $0$ or $8$ choices. 

\begin{conjecture}
Any nontrivial isomorphism $(\QLSr[8],\gamma)\simeq (\QLSralt[8],\gamma)$ at $d\geq 7$  has a pattern in the regular language described by
\[
\ts{0}^*\vee \ts{0}^*\ts{1}\vee \ts{0}^*(\ts{10}(\ts{00})^*)^*\ts{0}^*,
\]
and any such pattern is realized by three groups of nontrivial isomorphisms with 8 choices when the word  is of the form $\ts{0}\ts{0}^*$,  and one group of 8 at all other words in the language.
\end{conjecture}

Note that $\phi(5)=4=\phi(8)$ ($\phi$ denotes Euler's function); when $\phi(r)>4$  the observed systems are dramatically different, but we have not enough information to generate a conjecture from that.

Finally, we state the main conjecture.

\begin{conjecture}
Conditions (I)--(IV) are equivalent for all choices of $r$ and $d$.
\end{conjecture}

There is some evidence, notably \cite{lgcegdgrh:wcmg}, that Hazrat's conjectures hold for graphs of the kind we are considering, but at the level of generality needed here nothing is known with certainly. It is known (cf.\ \cite[Theorem 14.8]{segrerapws:ccuggs}) that stable isomorphism of quantum lens spaces implies exact isomorphism non-equivariantly, but this is only a necessary condition for (III)$\Longrightarrow$(I) to hold. It is not true that stabilized equivariant isomorphism implies exact equivariant isomorphism for such graphs, indeed easy examples can be found for the meteor graphs considered in  \cite{lgcegdgrh:wcmg}, so this would have to be a consequence of the rigid structure of quantum lens spaces.

\bibliographystyle{alpha}
\bibliography{/Users/vhq645/Reference/cstar,/Users/vhq645/Reference/own,/Users/vhq645/Reference/sds}

\end{document}